\documentclass{article}
\usepackage{amsmath, amssymb, amsfonts, latexsym,mathrsfs}
\usepackage{color,graphicx,hyperref,geometry}
\usepackage{inputenc}
\usepackage{footnote}
\usepackage{float}
\usepackage{enumerate}
\usepackage{ulem,booktabs,multirow}
\usepackage{bm,physics,units}
\usepackage{tikz}
\usepackage{verbatim}

\renewcommand{\d}{\mathop{}\!\mathrm{d}} 
\newcommand{\D}{\mathrm{D}} 

\newcommand{\pt}{\partial}

\newcommand{\ve}{\varepsilon}

\newcommand{\<}{\left<}
\renewcommand{\>}{\right>}

\newcommand{\mcalP}{\mathcal{P}}

\newcommand{\BbbR}{\mathbb{R}}

\newcommand{\dist}{\operatorname{dist}}
\newcommand{\argmin}{\mathop{\arg\min}}

\renewcommand{\emph}{\textit}

\makeatletter 
\@ifundefined{tr}{
    \newcommand{\tr}{\operatorname{tr}}
}
\@ifundefined{divergence}{
    \newcommand{\divergence}{\operatorname{\rm div}}
}{
    \renewcommand{\divergence}{\operatorname{\rm div}}
}
\makeatother

\usepackage{amsthm}
\newtheorem{theorem}{Theorem}
\newtheorem{lemma}[theorem]{Lemma}
\newtheorem{proposition}[theorem]{Proposition}
\newtheorem{corollary}[theorem]{Corollary}
\theoremstyle{definition}
\newtheorem{remark}{Remark}
\newtheorem*{remark*}{Remark}

\allowdisplaybreaks

\title{A diffuse-interface Landau-de Gennes model for free-boundary problems in the theory of nematic liquid crystals
\footnote{LZ is supported by the National Natural Science Foundation of China (No. 12225102, T2321001, and 12288101). AM is supported by the University of Strathclyde New Professors Fund, the Humboldt Foundation and a Leverhulme Research Project Grant RPG-2021-401. This work is partially supported by the Royal Society Newton Advanced Fellowship NAF/R1/180178 awarded to AM and LZ. }}

\date{\today}
\author{
Dawei Wu\footnote{School of Mathematical Sciences, Peking University, Beijing 100871, China ({2201110052@pku.edu.cn})}
\and
Baoming Shi\footnote{School of Mathematical Sciences, Peking University, Beijing 100871, China ({ming123@stu.pku.edu.cn})}
\and
Yucen Han\footnote{Department of Mathematics and Statistics, University of Strathclyde, Glasgow G1 1XQ, United Kingdom ({yucen.han@strath.ac.uk})}
\and
Pingwen Zhang\footnote{School of Mathematics and Statistics, Wuhan University, Wuhan 430072, China; School of Mathematical Sciences, Peking University, Beijing 100871, China ({pzhang@pku.edu.cn})}
\and
Apala Majumdar\footnote{Department of Mathematics and Statistics, University of Strathclyde, Glasgow G1 1XQ, United Kingdom ({apala.majumdar@strath.ac.uk})}
\and
Lei Zhang\footnote{Corresponding author. Beijing International Center for Mathematical Research, Center for Quantitative Biology, Center for Machine Learning Research, Peking University, Beijing 100871, China ({zhangl@math.pku.edu.cn})}
}

\begin{document}
\maketitle

\begin{abstract}
We introduce a diffuse-interface Landau-de Gennes free energy for nematic liquid crystals (NLC) systems, with free boundaries, in three dimensions submerged in isotropic liquid, and a phase field is introduced to model the deformable interface. 
The energy consists of the original Landau-de Gennes free energy, three penalty terms and a volume constraint.
We prove the existence and regularity of minimizers for the diffuse-interface energy functional. We also prove a uniform maximum principle of the minimizer under appropriate assumptions, together with a uniqueness result for small domains.
Then, we establish a sharp-interface limit where minimizers of the diffuse-interface energy converge to a minimizer of a sharp-interface energy using methods from $\Gamma$-convergence.
Finally, we conduct numerical experiments with the diffuse-interface model and the findings are compared with existing works.
\end{abstract}



\section{Introduction}

Liquid crystals (LC) are an intermediate state between solids and liquids, in which positional order is partially or completely lost, but molecular anisotropy is still present \cite{de_gennes_physics_1974}.
The simplest LC phase is the \textit{nematic} phase for which the constituent molecules have no positional order, but prefer to align along certain locally preferred directions known as \textit{directors}.
Nematic liquid crystals (NLCs) are endowed with direction-dependent physical, optical and rheological properties \cite{sonin_pierre-gilles_2018}.
NLCs exhibit intriguing morphology because of their molecular anisotropy, especially when a droplet of undetermined shape interacts with a different material, about which people have conducted many experiments.
A special droplet morphology of interest is the nematic \textit{tactoid}, i.e. spindle-shaped droplets filled with NLC whose directors align with the surface, which tends to emerge during the nematic-isotropic phase transition induced by temperature \cite{kim_morphogenesis_2013,nastishin_optical_2005,tortora_self-assembly_2010} or around rod-shaped bacteria as protective sheaths \cite{tarafder_phage_2020}. 
NLC droplets dispersed in polymer (known in the industry as polymer-dispersed liquid crystal) also show special opto-electric properties that inspire advances in the display industry \cite{doane_polymer_1988,kitson_controllable_2002,spencer_zenithal_2010,west_phase_1988,zografopoulos_beam-splitter_2014}.
The study of free-boundary NLC has also evoked insights into deformable anisotropic materials in other disciplines of science \cite{leoni_defect_2017,maroudas-sacks_topological_2021,sawa_shape_2013}.
These examples demonstrate the vast potential of free boundary problems for NLCs in confinement.

We need two essential ingredients to construct mathematical models for NLC problems with free boundaries: 
a degree of freedom to describe the free boundary and a NLC order parameter, to describe the nematic directors or the long-range orientational ordering in the NLC phase.
For modelling NLCs, there are competing molecular-based and macroscopic models e.g.,  the Onsager (molecular) model based on orientational distribution function, the continuum Oseen-Frank (OF) model for uniaxial NLC phases and the NLC order parameter is a unit-vector field that models the single distinguished material direction, and the Landau-de Gennes (LdG) model that describes the NLC phase in terms of the $\vb Q$-tensor order parameter, with five degrees of freedom that can describe uniaxial and biaxial NLC phases (NLC phase with a primary and a secondary nematic director) \cite{han_microscopic_2015, wang_modelling_2021}. 
With regard to modelling free boundaries, we quote existing methods from the field of \textit{shape and topology optimization} \cite{delfour_shapes_2011}.
The popular methods include the surface mesh method \cite{brakke_surface_1992} that discretizes the free boundary with a finite-element mesh known as the interface segregating different phases, 
and the diffuse-interface method \cite{wang_phase_2010} that employs a smooth field (space-dependent function), known as the \textit{phase field}, to describe the phase separation and interfacial regions. 
In recent decades, there have been multiple diverse approaches to modelling free boundary-value problems for NLC systems \cite{adler_nonlinear_2023,debenedictis_shape_2016,forest_lcp_2012,geng_two-dimensional_2022,golovaty_ginzburg-landau-type_2019,kaznacheev_influence_2003,kim_morphogenesis_2013,li_kinetic_2011,lin_isotropic-nematic_2023,mata_ordering_2014,xing_morphology_2012,yue_diffuse-interface_2004}.
The common practice is to design energy functionals with respect to the shape and to the NLC order parameter, and find the stable configuration with energy minimization methods.
For instance, in \cite{adler_nonlinear_2023} the authors describe a deformable two-dimensional (2D) NLC confinement with a finite-element mesh, and their energy functional consists of the LdG free energy and a penalty function for Dirichlet boundary values; they numerically find that for sufficiently strong penalty factors, the optimal shape converges to a tactoid, consistent with experimental findings. 
The authors of \cite{debenedictis_shape_2016} study the same problem with the finite element method as well, but they choose the OF free energy instead and further penalize the energy functional with the perimeter of the free boundary; they numerically find tactoids as the optimal shape when the penalty factors are large enough. 
In \cite{yue_diffuse-interface_2004}, the authors use a phase field, $\phi$, to denote separation of a NLC and a Newtonian fluid with, $\phi=1$ for NLC and $\phi=-1$ for fluid respectively, and their energy density consists of a boundary anchoring term, a mixing energy density of Van der Waals-Cahn-Hilliard form \cite{cahn_free_1958} to penalize the sharp phase separation ($\phi\approx \pm 1$) and smoothness of $\phi$, and finally the weighted OF free energy density masked by $\frac{\phi+1}{2}$ so that it is only integrated over the NLC phase; the authors numerically minimize the energy and find multiple stable states, including one radially symmetric state and one slightly elliptical bipolar state reminiscent of the tactoid. 
In \cite{geng_two-dimensional_2022}, the authors apply the OF model to study the 2D NLC droplet in the sharp-interface case, and theoretically predict the spindle-shaped tactoid. \label{pg:add1} 
In general, these shape/topology optimization methods have been hugely successful for NLC free boundary-value problems and there is immense potential for further theoretical and numerical exploration. 


We propose a diffuse-interface model for NLC free boundary-value problems in 3D, for nematic regions submerged in an isotropic liquid, using the LdG framework to describe the NLC phase and a phase field to describe the deformable shapes respectively. Our approach is mainly inspired by \cite{dai_convergence_2018,garcke_numerical_2015,yue_diffuse-interface_2004}. The model is relatively new since we may be the first people to combine LdG theory with diffuse-interface models for NLC free boundary-value problems.
We choose the LdG model because it is the most comprehensive continuum model for NLCs to date, and is yet more computationally tractable (with fewer degrees of freedom) than the molecular-level Onsager model. 
The LdG $\vb Q$-tensor order parameter contains information about the macroscopic quantities of interest --- the nematic director(s) and the degree of orientational ordering about the director(s).
The LdG model is a variational model with a LdG free energy, defined in terms of the LdG $\vb Q$-tensor, and the LdG energy minimizers model the physically observable or experimentally relevant NLC configurations. 
The LdG theory can describe uniaxial and biaxial NLC phases, spatially varying degrees of orientational order, \textit{defects} of different dimensionalities (point, line or surface discontinuities of the director field \cite{kleman_defects_1989,virga_variational_1994}) including non-orientable defects of fractional degrees ($+1/2$ and $-1/2$ defects in 2 dimensions, arising from the intrinsic head-to-tail symmetry of the NLC phase), while the OF theory is limited to uniaxial NLC phases with constant degrees of orientational ordering and cannot describe non-orientable point defects. 
Therefore, the LdG theory has been exceptionally successful for describing structural transitions in confined NLC systems.
The LdG model has been widely studied in the mathematical literature, and there are multiple theoretical results on the qualitative properties of LdG energy minimizers and LdG defect structures \cite{canevari_order_2017,henao_symmetry_2012,majumdar_equilibrium_2010}. 
LdG solution landscapes have also been extensively studied on fixed domains in two and three-dimensions and such studies shed powerful insight into the correlations between shape, geometry, boundary effects and stable equilibria, see e.g. \cite{han_multistability_2023,han_reduced_2020,han_solution_2021,shi_hierarchies_2023,shi_nematic_2022,yin_construction_2020}. 
Therefore, we believe that the LdG model has huge potential for NLC free boundary-value problems.
We choose the diffuse-interface model to describe the deformable shape because the phase field can be defined on a fixed grid \cite{singer-loginova_phase_2008,wang_phase_2010} while a surface mesh requires a full finite-element mesh \cite{brakke_surface_1992}. The degrees of freedom are therefore considerably lower for the phase field, since it only involves a scalar function, but the finite-element mesh requires the spatial coordinates of all nodes and the adjacency relations between them. 
Moreover, the boundary of the deformable shape is identified by variations of the phase field, so topological changes can be implemented, e.g. one region splitting into two or two regions merging into one \cite{chen_efficient_2022,leoni_defect_2017,yue_diffuse-interface_2004}. On the contrary, the surface mesh cannot describe topological changes because the node adjacency relations are fixed.

Our work is presented as follows. 
In Section \ref{sec-dildg}, we define a diffuse-interface energy functional $E_\ve$ composed of the classic LdG free energy in three dimensions, a mixing energy in terms of the phase field $\phi$, an anchoring energy to account for anchoring conditions on the nematic-isotropic interface and a void energy, along with a volume constraint on the phase field $\phi$. The small parameter $\ve>0$ models the thickness of the diffuse interface.
In Section \ref{sec-min}, we establish the solvability of this model, i.e. existence of minimizers of $E_\ve$, with classical methods of calculus of variation \cite{evans_partial_2010}.
In Section \ref{sec-reg}, we discuss the problem of minimizing $E_\ve$ with a fixed phase field. We obtain uniform bounds for the energy minimizers, independent of $\ve$ and address technical challenges stemming from the penalty terms, and also demonstrate uniqueness of energy minimizers (with fixed $\phi$) for sufficiently small domains. The uniqueness results follow from convexity arguments in \cite{lamy_bifurcation_2014}. Our discoveries agree with well-known results for the classical LdG energy on a fixed region.
In Section \ref{sec-sil}, we present our main contribution --- the limit of $E_\ve$ as $\ve\to 0,$ known as the \textit{sharp-interface limit}, with $\Gamma$-convergence \cite{de_giorgi_convergence_1978} as the primary tool.
We prove that under appropriate assumptions, the convergence of minimizers of $E_\ve$ to the minimizer of a sharp-interface functional, $E_0$, defined over deformable regions, that also models the separation of nematic and isotropic phases. The technical challenges originate from the limits of the anchoring energy, and that the mixing and the anchoring energy are absorbed as a single boundary integral in the sharp-interface energy.
In Section \ref{sec-num}, we conduct numerical experiments on the minimization of a reduced version of $E_\ve$ on a two-dimensional domain. We adjust hyperparameters to observe structural changes in the energy minimizer as a function of the penalty parameters, and compare them with existing results on the literature, particularly shape transitions from two-dimensional NLC-filled discs to NLC tactoids as a function of the hyperparameters. The numerical work illustrates the computational efficiency of our model and leads to several generalizations. We conclude with some perspectives in Section~\ref{sec:conclusion}.

\section{Diffuse-interface Landau-de Gennes model} \label{sec-dildg}

Let $\Omega\subset \BbbR^3$ be a three-dimensional (3D) region with Lipschitz boundary, where a blob of nematic liquid crystals is surrounded by an isotropic liquid. The nematic and isotropic phases are separated by the nematic-isotropic (N-I) interface.
We propose a \textit{diffuse-interface LdG energy}
\begin{equation} \label{ldg-phf}
\begin{aligned}
     E_\ve = E^{\rm LdG} + \omega_p E^{\rm mix}_\ve + \omega_a E^{\rm anch}_\ve + \omega_v E^{\rm void},
\end{aligned}
\end{equation}
where $\omega_p,\omega_v,\omega_a$ are positive weights of the competing energy terms, and $\ve>0$ is a positive parameter in the second and third terms respectively.
The physical dimensions of $\omega_p,\omega_v,\omega_a$, which will be given later, ensure that all terms on the RHS of \eqref{ldg-phf} have the dimension $\unit{N\cdot m}$, i.e. the dimension of energy.
The diffuse-interface LdG energy has four contributions: (a) LdG free energy of the nematic phase, (b) mixing energy or interfacial energy associated with the N-I interface, (c) the anchoring energy on the N-I interface that promotes tangential anchoring of the nematic molecules on the N-I interface and (d) void energy (penalty) of the isotropic phase.

\paragraph{(a) LdG energy}
We work in the LdG framework and describe the nematic phase by the tensor variable $\vb Q=((Q^{ij}))_{3\times 3}$, which is a $3\times 3$ symmetric traceless matrix. The eigenvectors of the LdG $\vb Q$-tensor model the nematic directors and the corresponding eigenvalues contain information about the degree of orientational ordering about the directors \cite{de_gennes_physics_1974}. Nematics are broadly classified as follows: the isotropic phase is modelled by $\vb Q=\vb 0$ so that there is no orientational ordering; $\vb Q$ is \textit{uniaxial} if $\vb Q$ has two degenerate non-zero eigenvalues and there is a single distinguished nematic director corresponding to the eigenvector with the non-degenerate eigenvalue; $\vb Q$ models a \textit{biaxial} phase if there are three distinct eigenvalues, and hence, a primary and secondary nematic director.
A uniaxial $\vb Q$-tensor can be written as
\[\vb Q=s\qty(\vb{n}\otimes\vb{n}-\frac13\vb I),\]
with $s\in\BbbR$ and $\vb{n}\in\mathbb{S}^2$ is the nematic director or the eigenvector corresponding to the non-degenerate eigenvalue. 
The term $E^{\rm LdG}$ in \eqref{ldg-phf} is the classical LdG free energy \cite{de_gennes_physics_1974}
\begin{equation} \label{E-ldg}
    E^{\rm LdG}[\vb Q]=\int_\Omega (F_{el} + F_b) \d x,
\end{equation}
with the one-constant elastic energy density
\begin{equation} \label{ldg-ela}
F_{el} = \frac{L}{2} |\grad\vb{Q}|^2,
\end{equation}
where $L>0$ is the material-dependent elastic constant, $|\grad\vb{Q}|^2\triangleq \sum_{i,j,k} |Q^{ij}_{x_k}|^2$. The bulk energy density is given by
\begin{equation}\label{ldg-bulk}
F_b(\vb Q) = \frac{A}{2} \tr\vb Q^2 - \frac{B}{3} \tr\vb Q^3 + \frac{C}{4} \qty(\tr\vb Q^2 )^2, 
\end{equation}
where $A<0$ is a re-scaled temperature and $B,C>0$ are material-dependent constants. We work with $A<0$ so that the minimizers of $F_b$ are a continuum of uniaxial $\vb Q$-tensors $\vb{Q}^*=s_+\qty(\vb{n}\otimes\vb{n}-\frac13\vb{I})$, with arbitrary unit vector $\vb{n}\in\mathbb{S}^2$ and
\begin{equation} \label{s-plus}
    s_+=\frac{B+\sqrt{B^2-24AC}}{4C}.
\end{equation}
The dimensions of $A,B,C$ are $\unit{N\cdot m^{-2}}$, and that of $L$ is $\unit{N}$, so that $E^{\rm LdG}$ has the dimension $\unit{N\cdot m}$.
The one-constant elastic energy density \eqref{ldg-ela} is analytically attractive and physically relevant for a large class of liquid crystal materials, but more general forms exist \cite{han_microscopic_2015,mottram_introduction_2014}.

\paragraph{(b) Mixing energy}
The term $E^{\rm mix}_\ve$ in \eqref{ldg-phf} is the mixing energy. It is the well-known \textit{Van der Waals-Cahn-Hilliard energy functional} \cite{cahn_free_1958,van_der_waals_thermodynamic_1979} for the binary mixture of two phases, in this case the nematic and isotropic phases respectively.
\begin{equation} \label{E-per}
    E^{\rm mix}_\ve[\phi]=\int_\Omega \left[ \ve |\nabla\phi|^2 + \ve^{-1} \phi^2(1-\phi)^2 \right] \d x,
\end{equation}
where $\phi$ is a phase field.
$\ve$ has the length dimension $\unit{m}$, so the term $E^{\rm mix}_\ve$ has dimension $\unit{m^2}$. We give the weight $\omega_p$ the dimension $\unit{N\cdot m^{-1}}$, so that the term $\omega_p E^{\rm mix}$ in \eqref{ldg-phf} has the energy dimension $\unit{N\cdot m}$.
Informally speaking, the continuous double-well function $W(\phi)=\phi^2(1-\phi)^2$ drives the separation of $\Omega$ into subdomains: the nematic subdomain with $\phi \approx 1$, the isotropic subdomain with $\phi\approx 0$, while the gradient term $|\nabla\phi|^2$ smoothens out $\phi$ to form a \textit{diffuse interface} whose thickness is represented by $\ve>0$, the \textit{capillary width} (see Figure \ref{fig:phasefield}).
In the $\ve \to 0$ limit, the Van der Waals-Cahn-Hilliard energy approximates the perimeter of the interface in the sense of $\Gamma$-convergence (the sharp interface limit) \cite{modica_gradient_1987}, so that it models the interfacial tension as is common in numerous physical models \cite{julicher_shape_1996,lowengrub_phase-field_2009}. 

\begin{figure}
    \centering
    \begin{tikzpicture}
    \filldraw[draw=black,fill=red!50!lightgray] (-2,-2) rectangle(2,2);
    \fill[fill=red!50!lightgray!99!white] (0,0) circle[radius=1.25];
    \fill[fill=red!50!lightgray!98.2!white] (0,0) circle[radius=1.2];
    \fill[fill=red!50!lightgray!95.3!white] (0,0) circle[radius=1.15];
    \fill[fill=red!50!lightgray!88.1!white] (0,0) circle[radius=1.1];
    \fill[fill=red!50!lightgray!73.1!white] (0,0) circle[radius=1.05];
    \fill[fill=red!50!lightgray!50!white] (0,0) circle[radius=1];
    \fill[fill=red!50!lightgray!26.9!white] (0,0) circle[radius=.95];
    \fill[fill=red!50!lightgray!11.9!white] (0,0) circle[radius=.9];
    \fill[fill=red!50!lightgray!4.7!white] (0,0) circle[radius=.85];
    \fill[fill=red!50!lightgray!1.8!white] (0,0) circle[radius=.8];
    \fill[fill=white] (0,0) circle[radius=.75];
    \draw[dashed,thick] (0,0) circle[radius=1];
    \draw (0,.3) node[color=black]{$\phi= 1$};
    \draw (0,1.5) node{$\phi= 0$};
    \draw[thick] (0,-.9)--(1.1,-.9);
    \draw[thick] (0,-1.1)--(1.1,-1.1);
    \draw[->] (1,-1.3)--(1,-1.1);
    \draw[->] (1,-.7)--(1,-.9);
    \draw (1,-1) node[anchor=west]{$O(\ve)$};
    \draw (.8,.6) -- (1.2,1.2) node[anchor=south west]{$\pt D$};
\end{tikzpicture}
    \caption{Schematic figure of N-I mixing. As colour change from bright to dark (white to greyish-red in colour), phase changes from nematic ($\phi=1$) to isotropic ($\phi=0$) via diffuse N-I interface with thickness $O(\ve)$ surrounding the surface $\partial D$.} \label{fig:phasefield}
\end{figure}
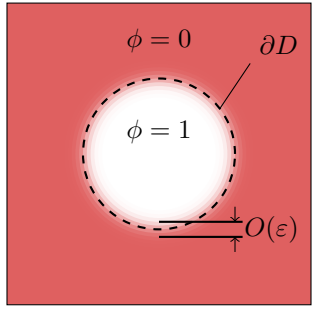

\paragraph{(c) Anchoring energy}
The term $E^{\rm anch}$ in \eqref{ldg-phf} enforces tangential anchoring \cite{de_gennes_physics_1974,virga_variational_1994} on the N-I interface and is defined by
\begin{equation} \label{E-anch}
    E^{\rm anch}_\ve[\vb Q,\phi]=\int_\Omega \ve \left| \left(\vb{Q}(x)+\frac{s_+}{3}\vb I\right) \nabla\phi \right|^2\d x,
\end{equation}
with $s_+$ defined inf \eqref{s-plus}.
$E^{\rm anch}_\ve$ also has dimension $\unit{m^2}$, and the weight $\omega_a$ has dimension $\unit{N\cdot m^{-1}}$, so that the term $\omega_a E^{\rm anch}$ appearing in \eqref{ldg-phf} has the energy dimension $\unit{N\cdot m}$.
$E^{\rm anch}_\ve$ is minimized when $\vb Q \nabla\phi=-\frac{s_+}{3}\nabla\phi$, i.e. when $\nabla\phi$ is an eigenvector of $\vb Q$ with negative eigenvalue $-\frac{s_+}{3}$. The vector $\nabla \phi$, approximates the normal vector to the diffuse interface and this constraint requires that the leading nematic director (eigenvector of $\vb Q$ with the largest positive eigenvalue) is orthogonal to $\nabla \phi$ or tangent to the N-I interface. 
To keep this energy bounded in the ($\ve \to 0$) sharp-interface limit, the factor $\ve$ is needed in \eqref{E-anch}. One can verify this fact by checking from Figure \ref{fig:phasefield} that $|\nabla\phi|=O(\ve^{-1})$ and that the size of its support is $O(\ve)$, making the integral in \eqref{E-anch} approximately $ \ve O(\ve^{-2})\cdot O(\ve)=O(1)$ as required, in the $\ve \to 0$ limit.
This weak anchoring energy has been applied elsewhere e.g., \cite{hu_disclination_2016} and \cite{shi_multistability_2024}.

\paragraph{(d) Void energy}
The term $E^{\rm void}$ in \eqref{ldg-phf} penalizes the isotropic phase ($\phi\approx 0$) surrounding the nematic phase and is given by
\begin{equation} \label{E-void}
    E^{\rm void}[\vb Q,\phi]= \int_\Omega \frac12(1-\phi)^2 |\vb Q|^2 \d x,
\end{equation}
where $|\vb Q|^2=\sum_{i,j} |Q^{ij}|^2$ is the Frobenius norm.
$E^{\rm void}$ has dimension $\unit{m^3}$, and the weight $\omega_v$ has dimension $\unit{N\cdot m^{-2}}$, so that the term $\omega_v E^{\rm void}$ in \eqref{ldg-phf} has the energy dimension $\unit{N\cdot m}$.
It is inspired by similar penalty terms in the studies of phase separation in fluids \cite{chen_efficient_2022,garcke_numerical_2015}.

We impose the volume constraint 
\[ \int_\Omega \phi\d x = V_0, \]
to conserve the mass of NLCs with $0<V_0 <|\Omega|$ ($|\Omega|$ is the volume of $\Omega$),
and impose homogeneous Dirichlet boundary conditions
\[ \vb{Q}|_{\pt\Omega}=\vb 0,\quad \phi|_{\pt\Omega}=0,\]
to model the isotropic phase surrounding the nematic phase.

The energy \eqref{ldg-phf} can be nondimensionalized. Let $\bar{x}=\lambda^{-1}x$, for all $x\in\Omega$, where $\lambda$ is a characteristic length scale associated with $\Omega$. After substituting $\bar x$ into the integrals and dividing the entire energy by $\lambda L$, we obtain
\begin{equation} \label{ldg-phf-nondim}
\begin{aligned}
    \bar{E}_{\bar{\ve}}[\vb{Q},\phi]
    &=\bar{E}^{\rm LdG}+\bar\omega_p \bar\lambda\bar{E}_{\bar\ve}^{\rm mix}+\bar\omega_a \bar\lambda\bar{E}_{\bar\ve}^{\rm anch}+\bar\omega_v \bar\lambda^2\bar{E}^{\rm void} \\
    &=\int_{\Omega^{(0)}} \qty[ \frac{1}{2} |\grad\vb{Q}|^2 + \bar\lambda^2 \bar F_b(\vb Q) ] \d \bar{x} + \bar\omega_p \bar\lambda \int_{\Omega^{(0)}} \qty[ \bar{\ve} |\nabla\phi|^2+ \frac{1}{\bar\ve} \phi^2(1-\phi)^2]\d \bar{x} \\
    &\quad +\bar\omega_a \bar\lambda \int_{\Omega^{(0)}} \bar{\ve}\left| \qty(\vb{Q}+\frac{s_+}{3}\vb I)\nabla\phi \right|^2 \d \bar{x} + \bar\omega_v \bar\lambda^2 \int_{\Omega^{(0)}} \frac12(1-\phi)^2 |\vb Q|^2 \d \bar{x}
\end{aligned}
\end{equation}
with the re-scaled domain $\Omega^{(0)}$, re-scaled nondimensional constants
\[ \bar{\lambda}^2 \triangleq \frac{\lambda^2 C}{L},\ 
\bar{\ve}\triangleq \frac{\ve}{\lambda},\ 
\bar{\omega}_p \triangleq \frac{\omega_p}{\sqrt{CL}},\
\bar{\omega}_v \triangleq \frac{\omega_v}{C},\ 
\bar{\omega}_a \triangleq \frac{\omega_a}{\sqrt{CL}},\]
and re-scaled LdG bulk energy $\bar F_b\triangleq \frac1C F_b$, subject to the re-scaled volume constraint
\begin{equation} \label{vol-con}
    \int_{\Omega^{(0)}} \phi\d \bar x = \bar V_0,
\end{equation} 
with $\bar V_0=\lambda^{-3}V_0$ satisfying $0<\bar V_0 < |\Omega^{(0)}|$.

In all subsequent discussion, we focus on the nondimensional functional $\bar E_{\bar\ve}$ \eqref{ldg-phf-nondim}, which is almost identical to \eqref{ldg-phf}, except for the additional $\bar\lambda$ in the leading coefficients and the bars above.
For notational simplicity, we rewrite $\Omega^{(0)}$ as $\Omega$ and drop bars over all constants from \eqref{ldg-phf-nondim} (note: the bar over $\bar F_b$ is kept in order to differ from the classical bulk energy $F_b$ \eqref{ldg-bulk}).
The four terms of \eqref{ldg-phf-nondim} are identified with the LdG energy $E^{\rm LdG}$ \eqref{E-ldg}, the mixing energy $E^{\rm mix}_\ve$ \eqref{E-per}, the anchoring energy $E^{\rm anch}_\ve$ \eqref{E-anch} and the void energy $E^{\rm void}$ \eqref{E-void} respectively.
We denote the admissible space of the order parameters $(\vb Q,\phi)$ by
\begin{equation} \label{admissible-space}
    \mathscr{A}=\qty{ (\vb Q,\phi):
        \vb Q \in H_0^1(\Omega;\mathcal{S}_0),\ 
        \phi \in H_0^1 (\Omega),\
        \int_\Omega\phi\d x=V_0
        }
\end{equation}
where $H_0^1$ is the Sobolev space \cite{adams_sobolev_2003,evans_partial_2010,gilbarg_elliptic_1977}
\begin{equation}
H_0^1(\Omega)=\qty{u\in L^2(\Omega): \int_\Omega [u^2+|\nabla u|^2] \d x<\infty,\ u|_{\pt\Omega}=0},
\end{equation}
$H_0^1(\Omega;V)$ is the Sobolev space with function values in a normed vector space $V$,
and finally,
\begin{equation} \label{trace0}
    \mathcal{S}_0=\{\vb Q\in\BbbR^{3\times 3}: \vb Q^T=\vb Q, \tr \vb Q=0\}
\end{equation}
is the space of symmetric traceless $3\times 3$ matrices.

If $(\vb Q,\phi)$ is a critical point of the functional \eqref{ldg-phf-nondim}, then this pair is a weak solution to the corresponding Euler–Lagrange equations:
\begin{equation} \label{ldgphf-el-Q}
    \begin{aligned}
        \triangle\vb{Q} 
        &= \lambda^2 \qty[ \qty(\frac{A}{C}+\tr\vb Q^2)  \vb Q - \frac{B}{C} \mathscr{P}_0(\vb Q^2)]
        + \omega_v\lambda^2 (1-\phi)^2 \vb Q \\
        &\quad
        +\omega_a\lambda \ve \mathscr{P}_0 \qty[(\nabla\phi \otimes\nabla\phi )\qty(\vb Q+\frac{s_+}{3}\vb I) + \qty(\vb Q+\frac{s_+}{3}\vb I)(\nabla\phi\otimes\nabla\phi)],
    \end{aligned}
\end{equation}
and
\begin{equation} \label{ldgphf-el-phi}
    \triangle\phi + \frac{\omega_a}{\omega_p}\divergence\qty[\qty(\vb Q+\frac{s_+}{3}\vb I)^2 \nabla\phi] = \frac{1}{\ve^2} \phi(\phi-1)(2\phi-1) + \frac{\omega_v\lambda}{2\omega_p\ve} (\phi-1)|\vb Q|^2 + \xi,
\end{equation}
where the projection operator is
\begin{equation} \label{trace0-proj}
    \mathscr{P}_0(\vb A)=\vb A-\frac{\tr \vb A}{3}\vb I,
\end{equation}
maintaining the traceless constraint on $\vb Q$,
and $\xi\in\BbbR$ is a Lagrangian multiplier originating from the volume constraint \eqref{vol-con}.
For a detailed calculation of \eqref{ldgphf-el-Q} and \eqref{ldgphf-el-phi}, please refer to Appendices \ref{sec-el}.

\section{Existence of minimizers} \label{sec-min}
We use the direct method in the calculus of variations \cite{evans_partial_2010} to prove the existence of a minimizer for the diffuse-interface LdG functional.

\begin{lemma} \label{ldgphf-wlsc}
The functional $E_\ve$ \eqref{ldg-phf-nondim} is weakly lower semi-continuous (w.l.s.c.) on the Banach space $(\vb Q,\phi)\in H_0^1(\Omega;\mathcal{S}_0) \times H_0^1(\Omega)$, which means that if $\vb Q_k \rightharpoonup \vb Q,\phi_k\rightharpoonup \phi$ are weakly convergent sequences in $H_0^1$, 
\[\liminf_{k\to\infty} E_\ve[\vb Q_k,\phi_k] \ge E_\ve[\vb Q,\phi].\]
\end{lemma}

\begin{proof}
As \cite[\S 8.2,Theorem 1]{evans_partial_2010}, 
for an integral energy functional of the form
\[\int_\Omega F(\nabla \bm u, \bm u, x)\d x\]
where $\bm u: \Omega\to\BbbR^m$ and $F:\BbbR^{m\times n}\times \BbbR^m\times \Omega \to \BbbR$, 
if $F$ is uniformly bounded below and convex with respect to its first argument, then
the functional is {w.l.s.c.} with respect to $\bm u\in H^1(\Omega;\BbbR^m)$.
We look at the individual terms.

Examine the form \eqref{ldg-phf-nondim} with $\bm u=(\vb Q,\phi): \Omega\to \mathcal{S}_0\times \BbbR$. The energy density is bounded from below because all terms are positive except for the LdG bulk energy $\bar F_b=C^{-1}F_b$ \eqref{ldg-bulk}, which satisfies 
\begin{equation} \label{bulk-min}
    \bar F_b(\vb Q)\ge m_b >-\infty, \quad\forall\vb Q\in\mathcal{S}_0
\end{equation}
for some number $m_b\in\BbbR$ that can be explicitly computed \cite[eq. (41)]{mottram_introduction_2014}.
Moreover, the terms involving the gradients, which read
\[\frac{1}{2} |\grad\vb{Q}|^2 + \omega_p\lambda\ve|\nabla\phi|^2+\omega_a\lambda\ve \left| \qty(\vb{Q}+\frac{s_+}{3}\vb I)\nabla\phi \right|^2,\]
are convex with respect to $(\grad\vb Q,\nabla\phi)$ since they are positive-definite quadratic polynomials. Therefore, the energy density of $E_\ve$ \eqref{ldg-phf-nondim} satisfies the criteria above, and is w.l.s.c. over the space $H_0^1(\Omega;\mathcal{S}_0)\times H_0^1(\Omega)$.
\end{proof}

\begin{proposition} \label{min-exist}
The functional, $E_\ve$ in \eqref{ldg-phf-nondim}, has at least one global minimizer $(\vb{Q}^*_\ve,\phi^*_\ve)$ in the admissible set $\mathscr{A}$ specified in \eqref{admissible-space}.
\end{proposition}
\begin{proof}
The admissible set $\mathscr{A}$ \eqref{admissible-space} is nonempty since there always exists a smooth $\phi\in H_0^1$ such that $\int_\Omega \phi = V_0$.

Moreover, all the terms in $E_\ve$ are bounded from below (recall \eqref{bulk-min}), yielding the following coercivity condition:
\begin{equation} \label{ldgphf-coer}
    \begin{aligned}
        E_\ve[\vb{Q},\phi] &\ge \int_\Omega \left[ \frac{1}{2} |\grad\vb{Q}|^2+ \omega_p\ve |\nabla\phi|^2 + \lambda^2 m_b \right]\d x \\
    &=\frac{L}{2} \|\grad\vb Q\|_{L^2}^2 + \omega_p\ve \|\nabla\phi\|_{L^2}^2 + \lambda^2 m_b |\Omega|.
    \end{aligned}
\end{equation}

The weak lower semi-continuity of $E_\ve$ over the feasible set is guaranteed by Lemma \ref{ldgphf-wlsc} since the admissible set $\mathscr{A}$ is a closed subspace of $H_0^1(\Omega;\mathcal{S}_0)\times H_0^1(\Omega)$. 
Therefore, the existence of a global minimizer follows from the direct method in the calculus of variations.
\end{proof}

\begin{corollary} \label{min-exist-Q}
For each fixed $\phi\in H_0^1$, the w.l.s.c. and coercivity argument still applies with respect to the variable $\vb Q$. Hence, there exists a $\vb Q$-minimizer, $\vb Q_\ve\in H_0^1(\Omega;\mathcal{S}_0)$ such that
\[\vb Q_\ve=\vb Q_\ve(\phi)=\argmin_{\vb Q\in H_0^1(\Omega; \mathcal{S}_0)} E_\ve[\vb Q,\phi].\]
\end{corollary}

\section{Maximum principle on a fixed domain} \label{sec-reg}

In this section, we prove that for a fixed $\phi$, the critical points of $E_\ve$ with respect to $\vb Q$ (referred to as $\vb Q$-critical points above) are bounded in $L^\infty$ uniformly, as $\ve\to 0$.
The maximum principle and uniqueness results have been studied thoroughly in the LdG framework; see \cite{lamy_bifurcation_2014,majumdar_equilibrium_2010}, but results for the diffuse-interface model \eqref{ldg-phf} involve different techniques, especially when dealing with the anchoring penalty \eqref{E-anch}.

\subsection{Uniform maximum principle}
Before introducing the uniform maximum principle, we improve the regularity of $\vb Q_\ve$ with the $L^p$ theory of second-order elliptic equations \cite{gilbarg_elliptic_1977}. Let
\[W^{m,p}(\Omega)=\qty{u\in L^p(\Omega): \|u\|_{W^{m,p}}^p = \sum_{k=0}^m \int_\Omega |\nabla^k u|^p <\infty}\]
be the general Sobolev spaces.

\begin{lemma} \label{reg-coarse}
If $\phi\in W^{1,\infty}(\Omega)$, then a $\vb Q$-critical point, $\vb Q_{\ve}\in H_0^1(\Omega;\mathcal{S}_0)$ of $E_{\ve}[\cdot,\phi]$ \eqref{ldg-phf-nondim}, is in $W^{2,p}$ for all $p<\infty$ and in $C^{1,\alpha}$, for all $\alpha<1$.
Since $\vb Q_\ve$ is a weak solution, continuity means that $\vb Q_\ve$ is equal to a $C^{1,\alpha}$ function a.e..
\end{lemma}
\begin{proof}
By definition, the critical point $\vb Q_\ve$ is a solution of the equation \eqref{ldgphf-el-Q}.
Examining the right-hand side, we find that it is a cubic polynomial of $\vb Q$ (because of the bulk energy derivative), and its coefficients, excluding constants, are determined by $\phi,\nabla\phi$, which in turn are uniformly bounded by assumption.

If $\vb Q_\ve\in L^p(\Omega),$ then the RHS (right-hand side) of \eqref{ldgphf-el-Q} is in $L^{p/3}(\Omega)$.
Therefore, by the $L^p$ theory of elliptic equations \cite{gilbarg_elliptic_1977}, $\vb Q_\ve \in W^{2,\frac{p}{3}}$ if $p>3$.
This enables us to improve the regularity of $\vb Q_\ve$ through a chain of Sobolev embeddings.

In our case, we have $\vb Q_\ve\in H_0^1,$ which implies that $\vb Q_\ve\in L^6$ from the Sobolev embedding theorem in 3D. 
Using the arguments above iteratively, we obtain $\vb Q_\ve \in W^{2,p}$ for all $p<\infty$, and with one extra step of embedding, this yields $\vb Q_\ve \in C^{1,\alpha}$ for all $\alpha<1$.
\end{proof}
 
Under stronger smoothness assumptions, we prove that $\vb Q$-critical points of the nondimensional energy \eqref{ldg-phf-nondim} are \textit{uniformly} bounded. The bound is uniform in the sense that it is independent of the capillary width $\ve$.
\begin{theorem} \label{Q-uni-bound}
Let $\phi_\ve\in W^{1,\infty}(\Omega)$ be a phase field, with the following extra assumptions.
\begin{enumerate}[\rm(i)]
\item $\phi_\ve\in[0,1]$ for all $x$.
\item  \label{pg:ne1}
We assume that there exists a $C^2$ vector field $\vb n_\ve$ defined on $\Omega$ satisfying
\[\begin{cases}
    |\vb n_\ve|\le 1, \\
    \|\vb n_\ve\|_{C^2(\Omega)}\le M_N <\infty
\end{cases}\]
for some uniform constant $M_N$, such that within the diffuse interface
\[\Gamma_\ve=\overline{\{x\in\Omega: 0<\phi(x)<1\}}\]
it holds that
\begin{equation} \label{n-eps}
    \nabla\phi_\ve = |\nabla\phi_\ve|\vb n_\ve.
\end{equation}
\end{enumerate}
Suppose $\vb Q_\ve$ is a $\vb Q$-critical point of $E_{\ve}[\cdot,\phi_\ve]$ \eqref{ldg-phf-nondim}.
Then for any fixed $\lambda_0> \lambda$, 
there exists a uniform constant $M$ determined by $A,B,C,\omega_v,\lambda_0$
and the bound $M_N$, but independent of $\ve$, such that
\[ \sup_{x\in\Omega} |\vb Q_\ve| \le M \lambda^{-\frac23}.\]
The bound is given in the a.e. sense.
\end{theorem}

\begin{proof}
In the following proof, we drop the subscript $\ve$ for notational simplicity.

Take the dot product of $\vb Q+\alpha\vb{n}\otimes\vb{n}$ with the Euler-Lagrange equation \eqref{ldgphf-el-Q}, where $\alpha$ is a number yet to be determined. Recall the operator $\mathscr{P}_0$ as defined in \eqref{trace0-proj}.
\begin{equation} \label{el-dot-Q}
\begin{aligned}
    &(\vb Q+\alpha\vb{n}\otimes\vb{n}):\triangle \vb Q\\
    &= \lambda^2 \underbrace{(\vb Q+\alpha\vb{n}\otimes\vb{n}):\qty[ \qty(\frac{A}{C}+\tr\vb Q^2)  \vb Q - \frac{B}{C} \mathscr{P}_0(\vb Q^2)]}_{(I)}
    + \omega_v\lambda^2 (1-\phi)^2 \underbrace{(\vb Q+\alpha\vb{n}\otimes\vb{n}):\vb Q}_{(II)} \\
    &\quad
    +\omega_a\lambda \ve \underbrace{(\vb Q+\alpha\vb{n}\otimes\vb{n}):\mathscr{P}_0 \qty[(\nabla\phi \otimes\nabla\phi )\qty(\vb Q+\frac{s_+}{3}\vb I) + \qty(\vb Q+\frac{s_+}{3}\vb I)(\nabla\phi\otimes\nabla\phi)]}_{(III)}.
\end{aligned}
\end{equation}

For the LHS (left-hand side) of \eqref{el-dot-Q}, we compute
\begin{align*}
    (\vb Q+\alpha\vb{n}\otimes\vb{n}):\triangle \vb Q
    &=(\vb Q+\alpha\vb{n}\otimes\vb{n}):\triangle (\vb Q+\alpha\vb{n}\otimes\vb{n}) - (\vb Q+\alpha\vb{n}\otimes\vb{n}):\triangle(\alpha\vb{n}\otimes\vb{n})\\
&=\frac{1}{2} \triangle(|\vb Q+\alpha\vb{n}\otimes\vb{n}|^2) - |\bm\nabla(\vb Q+\alpha\vb{n}\otimes\vb{n})|^2 - \alpha(\vb Q+\alpha\vb{n}\otimes\vb{n}):\triangle(\vb{n}\otimes\vb{n}) . 
\end{align*}
We have used the elementary equality $\vb A:\triangle \vb A=\frac12 \triangle(|\vb A|^2) - |\grad\vb A|^2$ ($\vb A=\vb Q+\alpha\vb{n}\otimes\vb{n}$). It is well-defined in the weak sense because $\vb Q+\alpha\vb{n}\otimes\vb{n}\in C^0\cap H^1$ by the regularity estimate from Lemma \ref{reg-coarse}.

For the RHS of \eqref{el-dot-Q}, we compute the three terms separately. The first term is
\begin{align*}
    (I) &=(\vb Q+\alpha\vb{n}\otimes\vb{n}):\qty[\qty(\frac{A}{C}+\tr\vb Q^2 ) \vb Q -\frac{B}{C} \qty(\vb Q^2-\frac{\tr\vb Q^2}{3}\vb I)] \\
    &=\qty(\frac{A}{C}+\tr\vb Q^2)(\tr \vb Q^2+\alpha\vb n\cdot\vb{Qn}) - \frac{B}{C} \qty(\tr\vb Q^3+\alpha|\vb{Qn}|^2-\frac{\alpha\tr\vb Q^2}{3}).
\end{align*}
The second term equals
\[ (II)= \tr\vb Q^2 + \alpha\tr(\vb{n}\otimes\vb{n}\vb Q)
= \tr\vb Q^2 + \alpha\vb n\cdot\vb{Qn}.\]
We have extensively used the equality $\tr\vb{AB}=\tr\vb{BA}$.

Suppose first that $x\in\Gamma$, so by the assumption \eqref{n-eps}, we have $\nabla\phi(x)=|\nabla\phi(x)| \vb n(x)$. Then we write
\begin{align*}
   (III)
&=|\nabla\phi|^2 (\vb Q+\alpha\vb{n}\otimes\vb{n}):\mathscr{P}_0 \qty[(\vb{n}\otimes\vb{n}) \qty(\vb Q+\frac{s_+}{3}\vb I) + \qty(\vb Q+\frac{s_+}{3}\vb I)(\vb{n}\otimes\vb{n})] \\
&=2|\nabla\phi|^2 (\vb Q+\alpha\vb{n}\otimes\vb{n}):\mathscr{P}_0 \qty[(\vb{n}\otimes\vb{n}) \qty(\vb Q+\frac{s_+}{3}\vb I)]
\end{align*}
using $\vb{n}\otimes\vb{n}=\frac{1}{|\nabla\phi|^2}\nabla\phi\otimes\nabla\phi$ and that $\vb Q$ is symmetric. 
The image of $\mathscr{P}_0$ has zero trace and hence, the product is unchanged if we add $\frac{s_+}{3}\vb I$ to $\vb Q+\alpha\vb{n}\otimes\vb{n}$.
Let $\tilde{\vb Q}=\vb Q+\frac{s_+}{3}\vb I$ (so $\tr\tilde{\vb Q}=s_+$), and we get
\begin{align*}
    (III)
    &=2|\nabla\phi|^2(\tilde{\vb Q}+\alpha\vb{n}\otimes\vb{n}): \mathscr{P}_0((\vb{n}\otimes\vb{n})\tilde{\vb Q}) \\
    &=2|\nabla\phi|^2(\tilde{\vb Q}+\alpha\vb{n}\otimes\vb{n}):\qty[(\vb{n}\otimes\vb{n})\tilde{\vb Q}-\frac{\tr((\vb{n}\otimes\vb{n})\tilde{\vb Q})}{3} \vb I] \\
    &=2|\nabla\phi|^2 \qty[ |\tilde{\vb Q}\vb n|^2 +\frac{2\alpha-s_+}{3} (\vb n\cdot\tilde{\vb Q}\vb n)].
\end{align*}
Setting $\alpha=\frac{s_+}{2}$ eliminates the term $\frac{2\alpha-s_+}{3} (\vb n\cdot\tilde{\vb Q}\vb n)$, which leads us to
\[(III)=2 \left|\qty(\vb Q+\frac{s_+}{3}\vb I)\nabla\phi\right|^2,\quad x\in\Gamma.\]
If $x\notin \Gamma$, then $\phi$ is a constant in a neighbourhood of $x$, so $\nabla\phi(x)\equiv 0$ and the third term equals zero. Therefore, we have
\[(III) \ge 0.\]

Then, substituting the above in \eqref{el-dot-Q}, and removing all positive terms on the RHS, we get the following inequality satisfied by the function $|\vb Q+\alpha\vb{n}\otimes\vb{n}|^2$ on the entire $\Omega$. This is known as a \textit{subsolution condition}.
\begin{equation} \label{rho-eq}
    -\triangle\qty(|\vb Q+\alpha\vb{n}\otimes\vb{n}|^2) + 2T(\vb Q) \le 0.
\end{equation}
The function $T(\vb Q)$ is defined to be
\begin{equation} \label{T-Q}
\begin{aligned}
    T(\vb Q)&=T^{(1)}(\vb Q)+\lambda^2 T^{(2)}(\vb Q), \\
    T^{(1)}(\vb Q)&=\alpha (\vb Q+\alpha\vb{n}\otimes\vb{n}):\triangle(\vb{n}\otimes\vb{n}) \\
    T^{(2)}(\vb Q)
    &=(\tr\vb Q^2)^2 - \frac{B}{C}\tr\vb Q^3 + \alpha (\tr\vb Q^2)(\vb n\cdot\vb{Qn}) 
    +\frac{3A+B\alpha}{3C} \tr\vb Q^2 -\frac{B}{C}\alpha|\vb{Qn}|^2 \\
    &\quad + \alpha\qty(\frac{A}{C}+\omega_v(1-\phi)^2)\vb n\cdot\vb{Qn}.
\end{aligned}
\end{equation}

Despite the complexity of its expression, $T(\vb Q)$ is basically a fourth order polynomial in terms of $\vb Q$, with uniformly bounded coefficients and a strictly positive leading term, $\lambda^2 (\tr\vb Q^2)^2$.
Therefore, there exists a number $R_0$ depending on $A,B,C,\lambda,\omega_v$ and the $C^2$-norm, $M_N$, from assumption (ii), such that
\begin{equation} \label{R-zero}
    T(\vb Q)\ge 0, \quad \forall |\vb Q|\ge R_0.
\end{equation}
Note that $\alpha=\frac{s_+}{2}$ is also determined by $A,B,C$ according to \eqref{s-plus}.

Finally, we use some standard techniques as in \cite[App. 2]{lamy_bifurcation_2014} for an example.
Define the cut-off 
\[\psi_k = \max\qty{ |\vb Q+\alpha\vb{n}\otimes\vb{n}|^2-k, 0},\]
such that $\psi_k\in H_0^1(\Omega)$ for $k>\alpha^2$ because $\vb Q|_{\pt\Omega}\equiv 0$ and that $|\alpha\vb{n}\otimes\vb{n}|^2\le \alpha^2$.
By the chain rule, we also have
\begin{equation} \label{grad-psi-k}
    \nabla\psi_k=\begin{cases}
        \nabla(|\vb Q+\alpha\vb{n}\otimes\vb{n}|^2), & |\vb Q+\alpha\vb{n}\otimes\vb{n}|^2>k, \\
        0, & |\vb Q+\alpha\vb{n}\otimes\vb{n}|^2\le k,
    \end{cases}
\end{equation}
Multiplying both sides of \eqref{rho-eq} with $\psi_k$, integrating by parts and eliminating positive terms, we get an inequality
\begin{equation} \label{subsol}
    \int_\Omega \nabla(|\vb Q+\alpha\vb{n}\otimes\vb{n}|^2)\cdot\nabla\psi_k\d x + 2\int_\Omega \psi_k T(\vb Q)\d x \le 0,\quad k>\alpha^2.
\end{equation}
By \eqref{grad-psi-k}, the first term of the LHS of \eqref{subsol} is simply $\int_\Omega |\nabla\psi_k|^2.$
The second term of \eqref{subsol} is supported on the region where $\psi_k\neq 0$, i.e. $|\vb Q+\alpha\vb{n}\otimes\vb{n}|^2>k$.
By the triangle inequality of the Frobenius norm,
\[|\vb Q| \ge |\vb Q+\alpha\vb{n}\otimes\vb{n}| - |\alpha\vb{n}\otimes\vb{n}| = |\vb Q+\alpha\vb{n}\otimes\vb{n}| - \alpha.\]
We choose $k>(R_0+\alpha)^2$ to ensure $|\vb Q|>R_0$ whenever $\psi_k$ is nonzero. Thus, according to \eqref{R-zero}, $T(\vb Q)$ is positive on the support of $\psi_k$, and 
\[\int_\Omega |\nabla\psi_k|^2\d x \le 0 \Rightarrow \psi_k\equiv 0,\]
i.e., $|\vb Q+\alpha\vb{n}\otimes\vb{n}|$ is uniformly bounded by $k$ for all $k>(R_0+\alpha)^2$. Therefore, using the triangle inequality on $\vb Q$ and $\alpha\vb{n}\otimes\vb{n}$ again, we get
\begin{equation} \label{pie}
    \sup_{x\in\Omega} |\vb Q|\le R_0+ 2\alpha,
\end{equation}
which is the uniform bound we desired.

Evidently, the bound \eqref{pie} depends on the choice of $R_0$, which ensures the positivity of $T(\vb Q)$ for all $|\vb Q|>R_0$. Hence, to complete the proof, we also need to study the dependence of $R_0$ \eqref{R-zero} on $\lambda$.
Closely observing the expressions of $T^{(1)}$ and $T^{(2)}$ \eqref{T-Q} and noting that their coefficients are controlled by a uniform bound dependent on $A,B,C,\omega_v,M_N$ but not on $\lambda$, we get
\begin{gather*}
    T^{(1)}(\vb Q) \ge - C_1 (|\vb Q|+1), \\
    T^{(2)}(\vb Q) \ge \frac12 |\vb Q|^4 - C_2, 
\end{gather*}
with constants $C_1,C_2>0$ depending on $A,B,C,\omega_v,M_N$. $T^{(1)}$ is affine in $\vb Q$ and $T^{(2)}$ is degree-4 in $\vb Q$ with leading term $(\tr \vb Q^2)^2$.
In order that $T(\vb Q)\ge 0$, it suffices to force the lower bound to be positive, i.e.
\[ T^{(1)}+\lambda^2 T^{(2)} \ge \lambda^2\qty(\frac12 |\vb Q|^4 - C_2) - C_1 (|\vb Q|+1)\ge 0.\]
Multiply both sides by $\lambda^{\frac23}$, and this reduces to checking that
\[\frac{1}{2}|\lambda^{\frac23} \vb Q|^4 - C_1 |\lambda^{\frac23}\vb Q| \ge C_1\lambda^{\frac23}+C_2\lambda^{\frac83}.\]
The RHS can be bounded in terms of $C_1,C_2,\lambda_0$ for all $\lambda < \lambda_0$.
The LHS is a degree-4 polynomial of $|\lambda^{\frac23}\vb Q|$ as a whole. By the elementary property of polynomial functions, there exists $C_3>0$ depending on $C_1,C_2,\lambda_0$,
such that the inequality holds whenever $|\lambda^{\frac23}\vb Q|>C_3$ i.e. $|\vb Q|>C_3 \lambda^{-\frac23}$.
Therefore, we take $R_0= C_3\lambda^{-\frac23}$ in the uniform bound \eqref{pie} and get 
\[\sup_{x\in\Omega} |\vb Q|\le C_3 \lambda^{-\frac23}+ 2\alpha \le M \lambda^{-\frac23}, 
\quad \forall \lambda<\lambda_0\]
with $M=C_3+ 2\alpha\lambda_0^{\frac23}$ dependent on $A,B,C,L,\omega_v,\lambda_0,M_N$ as desired.
\end{proof}

\begin{remark} \label{pg:ne2}
We construct a typical example of $\vb n_\ve$.
Take $h(x)$, the signed distance function \cite{modica_gradient_1987} to a smooth region $D$ representing the droplet, such that the level set $\{h=0\}$ equals $\pt D$, that $|\nabla h|=1$ a.e., and that $h(x)$ is smooth in a neighbourhood of $\pt D$ as long as $\pt D$ is smooth \cite{gilbarg_elliptic_1977}. 
We let $\phi_\ve$ be an approximation to the $\tanh$ function, also known as the standing wave \cite{evans_phase_1992}:
\[\phi_\ve(x)=
\begin{cases}
    0, & h<-2\sqrt{\ve}, \\
    \text{linear connection}, & -2\sqrt{\ve}\le h\le -\sqrt{\ve}, \\
    \frac12\left( 1-\tanh\frac{h(x)}{\sqrt{2}\ve} \right) , & |h|<\sqrt{\ve}, \\
    \text{linear connection}, & \sqrt{\ve}\le h\le 2\sqrt{\ve}, \\
    1, & h>2\sqrt{\ve}.
\end{cases}\]
Such a sequence is known to approximate minimizers of the Van der Waals-Cahn-Hilliard energy (our mixing energy).
We can take $\vb n_\ve=-\zeta(x)\nabla h,$ where $\zeta(x)\ge 0$ is a smooth cutoff function that equals 1 near $\pt D$ and equals zero outside of a fixed neighbourhood of $\pt D$, within which $h$ is smooth. Since the definition of $\vb n_\ve$ is not related to $\ve$, its $C^2$ norm is obviously uniform with respect to $\ve$.
By definition, $\nabla\phi_\ve$ is supported on the region $\{|h|\le 2\sqrt{\ve}\}$, an area of size $O(\sqrt{\ve})$ near $\pt D$, so the support of $\nabla\phi_\ve$ is contained inside the region where $\zeta\equiv 1$ when $\ve\to 0$.
Since $\phi_\ve$ is a decreasing function of $h$, $\nabla\phi_\ve$ is parallel to $-\nabla h$.
Therefore, when $\ve\to 0$, the desired relation $\nabla\phi_\ve=\frac{|\nabla\phi_\ve|}{|\nabla h|}\cdot(-\nabla h)=|\nabla\phi_\ve| \vb n_\ve$ holds within the support of $\nabla\phi_\ve$.
\end{remark}

\subsection{Uniqueness on small domains}
An immediate consequence of the maximum principle is the uniqueness of the critical points when $\lambda$ is sufficiently small.
\begin{proposition} \label{min-unique}
Under the assumptions of Theorem \ref{Q-uni-bound}, there exists a $\lambda_1>0$ depending on $\Omega,A,B,C,L,\omega_v$ and the bound $M_N$,
such that for all $\lambda<\lambda_1$, the $\vb Q$-critical point of $E_{\ve}[\cdot,\phi]$ \eqref{ldg-phf-nondim} is unique.
\end{proposition}
\begin{proof}
Following the approach of Lemma 8.2 from \cite[App. 2]{lamy_bifurcation_2014}, we only need to prove the strict convexity of the energy functional $E_{\ve}$ on the closed convex set
\[X_\lambda = \left\{\vb Q\in H_0^1(\Omega;\mathcal{S}_0): \sup_{x\in\Omega} |\vb Q| \le M\lambda^{-\frac23}\right\},\]
where $M=M(A,B,C,L,\omega_v,\lambda_0,M_N)$ is the constant stated in Theorem \ref{Q-uni-bound}, with $\lambda_0=1$. 
By definition, $X_\lambda$ contains all critical points of the energy. We highlight the dependence on $\lambda$ since the bound changes as $\lambda\to 0$.
It suffices to check that
\[ \frac{E_{\ve}[\vb Q+\vb h,\phi] + E_{\ve}[\vb Q-\vb h,\phi]}{2} \ge E_{\ve}\qty[\vb Q,\phi] \]
for any $\vb Q\in X_\lambda$ and $\vb h\in H_0^1(\Omega;\mathcal{S}_0)$, and that the strict inequality holds when $\vb h\neq\vb 0$.
Since the void energy and the anchoring energy are positive-definite quadratic polynomials with respect to $\vb Q$, they are both convex, and one needs to check the LdG energy alone:
\[ E^{\rm LdG}[\vb Q]=\int_{\Omega} \qty[\frac12 |\nabla\vb Q|^2 + \lambda^2\bar F_b(\vb Q)]\d x.\] The mixing energy is fixed by the fixed choice of $\phi$, for $\vb Q$-critical points of $E_{\ve}[\cdot,\phi]$.

Then, strong convexity of $E^{\rm LdG}$ is equivalent to the positivity of the second-order central difference
\[ \Delta(\vb Q,\vb h)=E^{\rm LdG}[\vb Q+\vb h] +  E^{\rm LdG}[\vb Q-\vb h]-2  E^{\rm LdG}[\vb Q] \]
for all $\vb Q\in X_\lambda$ and $\vb h \in H_0^1(\Omega;\mathcal{S}_0)\setminus\{\vb 0\}$. From the expression of $E^{\rm LdG}$, we compute
\[\Delta(\vb Q,\vb h)=\int_{\Omega} |\grad\vb h|^2
    +\lambda^2\int_{\Omega} \qty[\frac{A}{C} \tr\vb h^2
    -\frac{2B}{C}\tr(\vb {Qh}^2) + 2(\tr\vb{Qh})^2+\tr\vb Q^2 \tr\vb h^2+\frac12(\tr\vb h^2)^2].
\]
Since $|\vb Q|\le M\lambda^{-\frac23}$, we estimate a lower bound of $\Delta(\vb Q,\vb h)$ by removing all positive terms and amplifying $\vb Q$ to its largest norm. We get
\begin{align*}
    \Delta(\vb Q,\vb h) &\ge \int_{\Omega}|\grad\vb h|^2\d x - \lambda^2\int_{\Omega} \qty[\frac{|A|}{C} \tr\vb h^2 + \frac{2B}{C}\cdot M\lambda^{-\frac23} \tr \vb h^2]\d x \\
    &\ge \int_{\Omega}|\grad\vb h|^2 \d x- M' \lambda^{\frac43}  \int_{\Omega} |\vb h|^2\d x,
\end{align*}
where the constant $M'$ is scaled from $M$ by a constant related to $\frac{|A|}{C},\frac{B}{C}$. 
We note that the factor $\lambda^2$ absorbs the unbounded $\lambda^{-\frac23}$ (as $\lambda \to 0$) from the uniform maximum value introduced in Theorem \ref{Q-uni-bound}.

Poincar\'e's inequality \cite{gilbarg_elliptic_1977} implies that
\[\int_{\Omega}|\grad\vb h|^2 \d x \ge \frac{1}{C_4(\Omega)} \int_{\Omega}|\vb h|^2\d x,\]
so there exists a $\lambda_1=\min\{\lambda_0,(C_4 M')^{-\frac34}\}$ such that when $\lambda<\lambda_1$, the quantity $\Delta(\vb Q,\vb h)$ is always strictly positive for all $\vb h\neq\vb 0$.
Therefore, $\bar E^{LdG}$ is strictly convex on $X_\lambda$, and the uniqueness of $\vb Q$-critical points follows naturally.
\end{proof}

\section{Sharp-interface limit} \label{sec-sil}
The limiting behaviour of $E_\ve[\vb Q,\phi]$ \eqref{ldg-phf-nondim} in the $\ve\to 0$ limit, is known as the \textit{sharp-interface limit}.
A famous result of the sharp-interface limit concerns the Van der Waals-Cahn-Hilliard functional (as used in our mixing energy \eqref{E-per}), which states that it converges to the interface perimeter, as a $\Gamma$-limit in the $\ve \to 0$ limit \cite{modica_gradient_1987,sternberg_effect_1988}.
We are interested in the sharp-interface limit 
 because it guarantees that solving the continuous diffuse-interface problem (e.g. the Van der Waals-Cahn-Hilliard functional) approximates the difficult-to-compute sharp-interface problems (e.g. interface perimeter).

Our main result is:
\begin{theorem} \label{sharp-intf-limit}\label{pg:thm1}
Let $(\vb Q_\ve^*,\phi_\ve^*)\in \mathscr{A}$ be a minimizer of $E_\ve$ \eqref{ldg-phf-nondim} in its admissible space \eqref{admissible-space}. We make the following assumptions as $\ve\to 0$.
\begin{enumerate}[\rm(i)]  
    \item We assume that $\{\vb Q_\ve^*\}$ is a sequence of uniformly bounded continuous functions that $H^1$-weakly converge to $\vb Q_0\in C^0\cap H_0^1$.
    \item We assume that $\{\phi_\ve^*\}$ converges to the indicator function $I_{D_0}=\begin{cases}
        1, & x\in D_0, \\ 0 & x\notin D_0
    \end{cases}$ in $L^1$, where $D_0\subset\subset\Omega$ (abbreviation for $D_0\subset\Omega$ and $\pt D_0\subset\Omega$) is a smooth region such that $|D_0|=V_0$;
\end{enumerate}
Then, $(\vb Q_0, D_0)$ minimizes the energy functional
\begin{equation} \label{ldgphf-si}
    \begin{aligned}
    E_0[\vb Q,D]
    &=\int_\Omega \qty[ \frac{L}{2} |\grad\vb{Q}|^2 + \lambda^2 \bar F_b(\vb Q) ] \d x
    +\omega_v\lambda^2 \int_{D^c} \frac12 |\vb Q|^2 \d x \\
    &\quad
    +\frac{\omega_p\lambda}{3} \int_{\pt D} \qty[1+\frac{\omega_a}{\omega_p}\left| \qty(\vb Q+\frac{s_+}{3}\vb I)\bm\nu \right|^2 ]^{\frac12}\d S(x)
    \end{aligned}
\end{equation}
amongst all tensor functions $\vb Q\in C^0\cap H_0^1$ and regions $D\subset\subset \Omega$ with smooth boundary satisfying $|D|=V_0$, where $\bm\nu$ represents the unit normal vector of $\pt D$.
\end{theorem}
\begin{remark}
The sharp-interface limit \eqref{ldgphf-si} also comprises the LdG free energy and penalty factors on the void, the perimeter and the weak anchoring on $\pt D$. 
The first integral is identical to $E^{\rm LdG}[\vb Q]$, and the second is identical to $E^{\rm void}[\vb Q,I_{D}]$ \eqref{E-void} where $I_D$ is the indicator function of $D$.
Using the elementary inequality $\frac{1+|x|}{\sqrt 2} \le \sqrt{1+|x|^2} \le 1+ |x|,$ we find that
\[1\le\frac{
    \displaystyle 
    \mcalP(D) + \sqrt{\frac{\omega_a}{\omega_p}}\int_{\pt D} \left|\qty(\vb Q+\frac{s_+}{3}\vb I)\bm\nu \right|
}{
    \displaystyle
    \int_{\pt D} \qty[1+\frac{\omega_a}{\omega_p}\left|\qty(\vb Q+\frac{s_+}{3}\vb I)\bm\nu \right|^2 ]^{\frac12}\d S(x)
} \le \sqrt{2}.\]
Thus, the boundary integral is the sum of the perimeter and an $L^1$ anchoring penalty.
\end{remark}

\subsection{Proof by \texorpdfstring{$\Gamma$}{Γ}-convergence}

Theorem \ref{sharp-intf-limit} is proven in the framework of $\Gamma$-convergence, which is composed of an upper bound estimate and a lower bound estimate \cite{de_giorgi_convergence_1978}.

We re-write the diffuse-interface functional \eqref{ldg-phf-nondim} as follows:
\begin{equation} \label{ldgphf-alt}
    E_\ve[\vb Q,\phi]= E^{\rm LdG}[\vb Q] + \omega_v\lambda^2 E^{\rm void}[\vb Q,\phi] + \omega_p\lambda \int_\Omega \qty[\ve a_{\vb Q}(\nabla\phi) + \ve^{-1}W(\phi)]\d x
\end{equation}
where $W(s)=s^2(1-s)^2$ is the double-well function,
and the sharp-interface functional \eqref{ldgphf-si} as
\begin{equation} \label{ldgphfsialt}
    E_0[\vb Q,D]=E^{\rm LdG}[\vb Q]+\omega_v\lambda^2 E^{\rm void}[\vb Q,I_D] + \frac{\omega_p\lambda}{3}\int_{\pt D} |a_{\vb Q}(\bm\nu)|^{\frac12}\d S(x).
\end{equation}
Note that the expressions \eqref{ldgphf-alt} and \eqref{ldgphfsialt} share the LdG and void energy, as well as a quadratic form $a_{\vb Q}(\cdot):\BbbR^3\to\BbbR$ defined to be 
\begin{equation} \label{a-Q-xi}
    a_{\vb Q}(\xi) = |\xi|^2 + \frac{\omega_a}{\omega_p}\left|\qty(\vb Q+\frac{s_+}{3}\vb I)\xi\right|^2,\quad \xi\in\BbbR^3,
\end{equation}
suggesting the required connection between $E_\ve$ and $E_0$.

We first present a compactness result.
\begin{proposition} \label{ldgphf-comp} \label{pg:comp}

Work with the same assumptions as in Theorem \ref{sharp-intf-limit}. Then, for any sequence $\{(\vb Q_\ve,\phi_\ve)\}$:
\begin{enumerate}[\rm(a)]
    \item If $E^{\rm LdG}[\vb Q_\ve]$ is bounded for all $\ve$, then $\{\vb Q_\ve\}$ is weak precompact in $H_0^1$.
    \item If $E^{\rm mix}_\ve[\phi_\ve]$ is bounded for all $\ve$, then $\{\phi_\ve\}$ is precompact in $L^1$, and any cluster point of $\{\phi_\ve\}$ takes binary values $0$ and $1$.
    \item Denote by
    \begin{equation} \label{int-sqrt-W}
        \Phi(s)=\int_0^s\sqrt{W(s)}\d s.
    \end{equation}
    If $E^{\rm mix}_\ve[\phi_\ve]$ and $E^{\rm anch}_\ve[\vb Q_\ve,\phi_\ve]$ are bounded for all $\ve$, then $\nabla\Phi(\phi_\ve)\d x$ and $\vb Q_\ve\nabla\Phi(\phi_\ve)\d x$ are weak-$*$ precompact as Radon measures on $\Omega$.
    The weak-$*$ precompactness of a Radon measure sequence $\{\mu_n\}$ on $\Omega$ means that there exists a subsequence $\{\mu_{n_k}\}$ and another measure $\mu$, such that $\mu_{n_k}\overset{*}{\rightharpoonup}\mu$, i.e. for any function $f\in C_0^\infty(\Omega)$,
    \[\int_\Omega f(x)\d\mu_{n_k}(x) \to \int_\Omega f(x)\d\mu(x).\]
\end{enumerate}
Especially, if $E_\ve[\vb Q_\ve,\phi_\ve]$ is uniformly bounded for all $\ve$, then all of the above holds.
\end{proposition}

Then, we state the $\Gamma$-convergence of $E_\ve$ to $E_0$.
\begin{proposition} \label{ldgphf-gam-lim}
Work with the same assumptions as in Theorem \ref{sharp-intf-limit}, and the following properties hold for all continuous $\vb Q\in H_0^1$ and smooth regions $D\subset\subset\Omega$ satisfying $|D|=V_0$.
\begin{enumerate}[\rm(a)] \label{pg:gamprop}
\item For all sequences $\phi_\ve\to I_{D}$ ($I_D$ is the indicator function of $D$) in $L^1$, and $\vb Q_\ve\rightharpoonup \vb Q$ in $H_0^1$, with $\vb Q_\ve$ continuous and bounded by a constant independent of $\ve$,
\begin{equation} \label{ldgphf-liminf}
    \liminf_{\ve\to 0} E_\ve[\vb Q_\ve,\phi_\ve] \ge E_0[\vb Q,D];
\end{equation}
\item There exists a sequence $\phi_\ve\to I_{D}$ in $L^1$ satisfying the volume constraint \eqref{vol-con},  such that
\begin{equation} \label{ldgphf-limsup} 
    \limsup_{\ve\to 0} E_\ve[\vb Q,\phi_\ve] \le E_0[\vb Q,D].
\end{equation}
\end{enumerate}
\end{proposition}

The lower and upper limits of Proposition \ref{ldgphf-gam-lim} rely on the following claims.
\begin{itemize}
\item Claim (I): For all sequences $\phi_\ve\to I_D$ in $L^1$ and $\vb Q_\ve\to\vb Q$ in $L^1$,
\begin{equation} \label{gamlim-1-2-inf}
\liminf_{\ve\to 0} \qty( E^{\rm LdG}[\vb Q_\ve] + \omega_v\lambda^2 E^{\rm void}[\vb Q_\ve,\phi_\ve]) \ge 
E^{\rm LdG}[\vb Q] 
+\omega_v\lambda^2 E^{\rm void}[\vb Q,I_D]
\end{equation}
\item Claim (II): For all $\phi_\ve\to I_D$ a.e. with $0\le\phi_\ve\le 1$, it holds that
\begin{equation} \label{gamlim-1-2-sup}
    \lim_{\ve\to 0} \qty( E^{\rm LdG}[\vb Q] + \omega_v\lambda^2 E^{\rm void}[\vb Q,\phi_\ve]) = E^{\rm LdG}[\vb Q] +\omega_v\lambda^2 E^{\rm void}[\vb Q,I_D]
\end{equation}
\item Claim (III): For all sequences $\phi_\ve\to I_{D}$ in $L^1$ and $\vb Q_\ve\rightharpoonup \vb Q$ in $H_0^1$ with a uniform bound independent of $\ve$,
\begin{equation} \label{gamlim-3-4-inf}
    \liminf_{\ve\to 0} \int_\Omega \qty[\ve a_{\vb Q_\ve}(\nabla\phi_\ve) + \ve^{-1}W(\phi_\ve)] \d x 
    \ge \frac13 \int_{\pt D} |a_{\vb Q}(\bm\nu)|^{\frac12}\d S(x).
\end{equation}
\item Claim (IV): There exists a sequence $\phi_\ve\to I_{D}$ a.e. with $0\le\phi_\ve\le 1$ and the volume constraint \eqref{vol-con}, such that
\begin{equation} \label{gamlim-3-4-sup}
    \limsup_{\ve\to 0} \int_\Omega \qty[\ve a_{\vb Q}(\nabla\phi_\ve) + \ve^{-1}W(\phi_\ve)] \d x
    \le \frac13 \int_{\pt D} |a_{\vb Q}(\bm\nu)|^{\frac12}\d S(x).
\end{equation}
\end{itemize}
Obviously, Claims (I) and (II) concern the convergence of the first two terms of \eqref{ldgphf-alt} and \eqref{ldgphfsialt}, while Claims (III) and (IV) are about the final term.
For technical reasons, we note that the sequences in Claim (II) and (IV) have $\vb Q$ fixed, and that the sequences in Claim (IV) satisfy Claim (II). Claims (I), (II) and (III) are stated without the volume constraint \eqref{vol-con} on $\{\phi_\ve\}$, and are therefore stronger than necessary.

\subsection{Detailed proof}

We start with compactness.
\begin{proof}[Proof of Proposition \ref{ldgphf-comp}] \label{pg:comp-pf}
(a) Weak precompactness of $\{\vb Q_\ve\}$. We notice that the LdG elastic energy \eqref{ldg-ela} is a multiple of the $H^1$ norm of $\vb Q_\ve$, so the uniform boundedness of $E_\ve$ implies that
\[\|\vb Q_\ve\|_{H^1} \le M\]
for some uniform constant $M$. Since $H_0^1$ is a Hilbert space, the bounded sequence $\{\vb Q_\ve\}$ is always weak precompact \cite{evans_partial_2010}.

(b) Precompactness of $\{\phi_\ve\}$.
We recall that
\[E^{\rm mix}_\ve= \int_\Omega \left[ \ve |\nabla\phi_\ve|^2 + \ve^{-1} W(\phi_\ve) \right] \d x\]
is the well-known Van der Waals-Cahn-Hilliard energy \cite{cahn_free_1958,van_der_waals_thermodynamic_1979}.
The rest of the proof follows from a standard method (see e.g. \cite[Prop. 3]{sternberg_effect_1988}).

First, $W(s)=O(s^4)$ as $s\to\infty$, so the fact that $E^{\rm mix}_\ve$ is uniformly bounded implies that $\phi_\ve$ is uniformly bounded in $L^4$. Using an elementary inequality $a^2+b^2\ge 2ab$ to cancel out $\ve$, we get
\begin{equation}\label{a2b2-ge-2ab}
\begin{aligned}
    E_\ve^{\rm mix}[\phi_\ve] &\ge \int_\Omega 2\sqrt{\ve |\nabla\phi_\ve|^2\cdot \ve^{-1} W(\phi_\ve)}\d x \\
    &=2\int_\Omega |\nabla\Phi(\phi_\ve)|\d x,
\end{aligned}
\end{equation}
where the chain rule on $\Phi$ \eqref{int-sqrt-W} is used.
Hence, $\Phi(\phi_\ve)$ is uniformly bounded in $\mathrm{BV}(\Omega)$, the space of functions with bounded variation \cite{evans_measure_2015,giusti_minimal_1984}. By the compactness of $\mathrm{BV}$ in $L^1$, there exists a $L^1$ convergent subsequence $\{\Phi(\phi_{\ve_k})\}$. 
Since $\Phi$ is a strictly monotonic increasing function on $\BbbR$, $\phi_{\ve_k}=\Phi^{-1}(\Phi(\phi_{\ve_k}))$ converges in measure. But since $\{\phi_{\ve_k}\}$ is also uniformly bounded in $L^4$, the convergence is in $L^1$ as well.

Moreover, we notice that
\[\int_\Omega W(\phi_\ve)\d x \le \ve E^{\rm mix}_\ve \to 0,\ \text{as }\ve\to 0,\]
so the $L^1$ limit $\phi$ must satisfy that $W(\phi)=0,$ a.e., which means it only takes binary values $0$ and $1$.

(c) Weak-$*$ precompactness of the vector-valued Radon measures.
By \eqref{a2b2-ge-2ab}, we have known that $\nabla\Phi(\phi_\ve)$ are uniformly bounded in $L^1$. We regard $\nabla\Phi(\phi_\ve)$ as a vector-valued bounded Radon measure, i.e. a functional of continuous vector-valued functions $\bm g \in C_0(\Omega)$ defined through
\[\<\bm g,\nabla\Phi(\phi_\ve)\> \triangleq \int_\Omega \bm g\cdot \nabla\Phi(\phi_\ve)\d x.\]
By the well-known Riesz representation theorem \cite[Thm. 1.38]{evans_measure_2015}, Radon measures on $\Omega$ constitute the dual space of $C_0(\Omega)$, the space of continuous functions on $\Omega$ with compact support, for all bounded Lipschitz domains $\Omega$. $L^1(\Omega)$ equipped with its norm is a closed subspace of the Radon measures.
By Banach-Alaoglu theorem, bounded Radon measures are weak-$*$ compact, so $\{\nabla\Phi(\phi_\ve)\}$ is weak-$*$ precompact \cite[Thm. 1.41]{evans_measure_2015}.

Similarly, we take the anchoring density $\ve|(\vb Q_\ve+s_+\vb I/3)\nabla\phi_\ve|^2$ from $E_\ve^{\rm anch}$ and the double-well term $\ve W(\phi_\ve)$ from $E_\ve^{\rm mix}$ and perform the same operation as \eqref{a2b2-ge-2ab} to get the uniform boundedness of
\[M \ge \int_\Omega \left|\qty(\vb Q_\ve+\frac{s_+}{3}\vb I)\nabla\Phi(\phi_\ve)\right| \d x,\]
so $(\vb Q_\ve+s_+\vb I/3)\nabla\Phi(\phi_\ve)$ is also uniformly bounded in $L^1$. By the triangle inequality,
\[\vb Q_\ve \nabla\Phi(\phi_\ve)=\qty(\vb Q_\ve+\frac{s_+}{3}\vb I)\nabla\Phi(\phi_\ve) - \frac{s_+}{3}\nabla\Phi(\phi_\ve)\]
is uniformly bounded in $L^1$. We apply the same reasoning as before and complete the proof.
\end{proof} 

Claims (I) and (II) can be obtained by means of basic calculus.
\begin{proof}[Proof of Claim (I)]
Assume w.l.o.g. that 
\[\liminf_{\ve\to 0}\int_\Omega |\grad\vb Q_\ve|^2\d x<\infty,\]
for otherwise the conclusion is trivial.
Due to the weak compactness of $H_0^1$, we extract a subsequence denoted by $\{\vb Q_{\ve_k}\}$ such that
\[\lim_{k\to\infty}\int_\Omega |\grad\vb Q_{\ve_k}|^2\d x = \liminf_{\ve\to 0} \int_\Omega |\grad\vb Q_\ve|^2\d x \text{ and }\vb Q_{\ve_k}\rightharpoonup \vb Q \text{ in }H_0^1.\]
Note that since $\vb Q_\ve\to\vb Q$ in $L^1$ by assumption, the weak limit of $\vb Q_{\ve_k}$ in $H_0^1$ must also be $\vb Q$.
Hence, we have by the w.l.s.c. property \cite[sect. 8.2, Theorem I]{evans_partial_2010} that
\[\liminf_{\ve\to 0}\int_\Omega |\grad\vb Q_\ve|^2\d x = \lim_{k\to\infty} \int_\Omega|\grad\vb Q_{\ve_k}|^2\d x \ge \int_\Omega |\grad\vb Q|^2\d x.\]
Since the bulk energy $\bar F_b(\vb Q)$ is a degree-4 polynomial $\bar F_b(\vb Q_{\ve_k})\to \bar F_b(\vb Q)$ strongly by the compact Sobolev embedding $H_0^1\overset{c}{\hookrightarrow} L^4$ in 3 dimensions \cite{adams_sobolev_2003,gilbarg_elliptic_1977}.
Thus,
\[ \lim_{k\to\infty}\int_\Omega \bar F_b(\vb Q_{\ve_k})\d x =\int_\Omega \bar F_b(\vb Q)\d x.\]
For the void penalty $E^{\rm void}$ \eqref{E-void}, we once again extract a subsequence $\{\phi_{\ve_k}\}\to I_D$ (w.l.o.g. the same subsequence as before) such that
\[\lim_{k\to\infty} E^{\rm void}[\vb Q_{\ve_k},\phi_{\ve_k}] = \liminf_{\ve\to 0} E^{\rm void}[\vb Q_{\ve},\phi_\ve]\text{ and }\phi_{\ve_k}\to I_D,\text{a.e.}.\]
Recall the expression \eqref{E-void} of $E^{\rm void}$. It follows from Fatou's lemma \cite{rudin_real_1987} that
\begin{align*}
    \liminf_{\ve\to 0} E^{\rm void}[\vb Q_\ve,\phi_\ve]
    &= \lim_{k\to\infty} \int_\Omega \frac12 (1-\phi_{\ve_k})^2 |\vb Q_{\ve_k}|^2\d x \\
    &\ge \int_\Omega \liminf_{k\to\infty} \frac12(1-\phi_{\ve_k})^2 |\vb Q_{\ve_k}|^2\d x
    =E^{\rm void}[\vb Q,I_D].
\end{align*}
Adding the expressions above gives \eqref{gamlim-1-2-inf}.
\end{proof}
\begin{proof}[Proof of Claim (II)]
The same LdG energy is featured on both sides of \eqref{gamlim-1-2-sup}, so we only have to check the void energy. Since $\phi_\ve$ and $\vb Q\in C^0$ are uniformly bounded and since $\phi_\ve \to I_D$ a.e., it follows from the dominated convergence theorem that
\[\lim_{\ve\to 0} \int_{\Omega} \frac12(1-\phi_\ve)^2|\vb Q|^2\d x = \int_\Omega \frac12 (1-I_D)^2 |\vb Q|^2\d x,\]
which yields \eqref{gamlim-1-2-sup}, as required.
\end{proof}

Then, we prove Claim (III) with properties of Radon measures \cite{evans_measure_2015}, as well as the method of \cite[Prop. 1]{modica_gradient_1987}.

\begin{proof}[Proof of Claim (III)]  \label{pg:clm3-pf}
We first argue that we can assume $\phi_\ve$ are uniformly bounded in $[0,1]$.
Otherwise, we can define the truncated function
\[\tilde\phi_\ve(x)=\max\{\min\{\phi(x),1\},0\}=\begin{cases}
    0, & \phi(x)<0, \\
    \phi(x), &  0\le \phi(x)\le 1, \\
    1, & \phi(x)> 1.
\end{cases}\]
Recall that $I_D$ takes binary values 0 and 1, so obviously $|\tilde\phi_\ve(x) - I_D(x)| \le |\phi_\ve(x)-I_D(x)|$ for all $x$. Hence,
\[\|\tilde\phi_\ve-\phi_0\|_{L^1} \le \|\phi_\ve-\phi_0\|_{L^1} \to 0,\]
i.e. the sequence $\{\tilde\phi_\ve\}$ converges in $L^1$ as well.
In addition, it follows from the definition of the double-well function $W$ and the chain rule that
\[W(\tilde\phi_\ve) \le W(\phi_\ve), \ 
\nabla\tilde\phi_\ve = I_{\{0<\phi_\ve < 1\}} \nabla\phi_\ve,\]
so
\[\int_\Omega [\ve^{-1} W(\tilde\phi_\ve)+\ve a_{\vb Q_\ve}(\nabla\tilde\phi_\ve)]\d x \le \int_\Omega [\ve^{-1} W(\phi_\ve)+\ve a_{\vb Q_\ve}(\nabla\phi_\ve)]\d x.\]
Therefore, in order to prove Claim (III), it suffices to check that
\[\liminf_{\ve\to 0} \int_\Omega [\ve^{-1} W(\tilde\phi_\ve)+\ve a_{\vb Q_\ve}(\nabla\tilde\phi_\ve)]\d x\ge \frac13 \int_{\pt D} |a_{\vb Q}(\bm\nu)|^{\frac12}\d S,\]
with $\tilde\phi_\ve \to I_D$ in $L^1(\Omega)$ ($\ve\to 0$) and the extra assumption $0 \le \tilde\phi_\ve \le 1$.

Extract a subsequence $\{\vb Q_{\ve_k},\phi_{\ve_k}\}\triangleq\{\vb Q_k,\phi_k\}$ such that
\[\lim_{k\to\infty} \int_\Omega [\ve^{-1} W(\tilde\phi_k)+\ve a_{\vb Q_k}(\nabla\tilde\phi_k)]\d x
=\liminf_{\ve\to 0} \int_\Omega [\ve^{-1} W(\tilde\phi_\ve)+\ve a_{\vb Q_\ve}(\nabla\tilde\phi_\ve)]\d x <\infty.\]
By \eqref{ldgphf-alt}, the LHS of the above formula equals the sum of $E^{\rm mix}_{\ve_k}[\phi_k]$ and $E^{\rm anch}_{\ve_k}[\vb Q_k,\phi_k]$, so these terms are bounded with respect to $\ve$.
With the compactness result from Proposition \ref{ldgphf-comp}, we can assume further that the Radon measures $\nabla\Phi(\phi_k)$ and $\vb Q_k\nabla\Phi(\phi_k)$ are both weak-$*$ convergent by extracting another subsequence.
Since we also assume that $\phi_k\to I_D$ in $L^1$ (so $\Phi(\phi_k)\to \Phi(I_D)$ in $L^1$ by the Lipschitz continuity of $\Phi$ \eqref{int-sqrt-W}) and that $\vb Q_k\rightharpoonup \vb Q$ in $H_0^1$, the respective weak-$*$ limits of these measures must be
\begin{equation} \label{l1-h1w-then-rmws}
    \nabla\Phi(\phi_k) \overset{*}{\rightharpoonup} \D \Phi(I_D),\ 
\vb Q_k\nabla\Phi(\phi_k)\overset{*}{\rightharpoonup} \vb Q\D \Phi(I_D),
\end{equation}
where $\D$ stands for the weak derivative as a generalized function, as distinguished from $\nabla$.
\begin{proof}[Proof of \eqref{l1-h1w-then-rmws}]
It suffices to prove a component-wise result. That is, for all function sequences $f_n\rightharpoonup f$ in $H_0^1(\Omega)$ and $g_n\to g$ in $L^1(\Omega)$, satisfying further that $f_n,f\in C^0\cap H_0^1$, $f_n,f,g_n,g$ are uniformly bounded, and that $\pt_{x_i}g_n$ and $f_n \pt_{x_i} g_n$ both weak-$*$ converge to some Radon measure on $\Omega$, it must hold that
\[\pt_{x_i}g_n \overset{*}{\rightharpoonup} \pt_{x_i} g,\ 
f_n \pt_{x_i} g_n \overset{*}{\rightharpoonup} f \pt_{x_i} g.\]
Up to extracting a subsequence, we may assume further that $f_n\to f$ and $g_n\to g$ a.e..
We need to verify for each $\psi \in C_0^\infty(\Omega)$ that
\[\int_\Omega \psi \pt_{x_i}g_n\d x \to \int_{\Omega} \psi \D_{x_i} g\d x,\ 
\int_\Omega \psi f_n \pt_{x_i}g_n\d x \to \int_{\Omega} \psi f \D_{x_i} g\d x.\]
Using the definition for weak derivatives and integration by parts, we see that the condition is equivalent to
\[\int_\Omega g_n\pt_{x_i}\psi \d x \to \int_{\Omega} g \pt_{x_i}\psi\d x,\ 
\int_\Omega (f_n g_n\pt_{x_i}\psi + \psi g_n \pt_{x_i} f_n)\d x
\to \int_\Omega (f g\pt_{x_i}\psi + \psi g \pt_{x_i} f)\d x.\]
The first convergence is obviously correct, as $\pt_{x_i}\psi$ is smooth. 
By Lebesgue's dominated convergence theorem, since $f_ng_n\to fg$ a.e. and $\{f_ng_n\}$ is uniformly bounded by assumption,
\[\int_\Omega f_n g_n\pt_{x_i}\psi\d x
\to \int_\Omega f g\pt_{x_i}\psi\d x.\]
Since $\psi g_n\to \psi g$ a.e. and $\pt_{x_i} f_n\to \pt_{x_i} f$ weakly in $L^2$,
\[\int_\Omega \psi g_n\pt_{x_i} f_n\d x
\to \int_\Omega \psi g\pt_{x_i} f\d x.\]
Hence, the second convergence is also correct.
The correctness of \eqref{l1-h1w-then-rmws} is then established.
\end{proof}

Using the same elementary inequality technique as in \eqref{a2b2-ge-2ab}, we compute that
\begin{equation} \label{liminf-ge-1}
\begin{aligned}
    \int_\Omega \qty[\ve a_{\vb Q_k}(\nabla\phi_k) + \ve^{-1}W(\phi_k)] \d x &\ge \int_\Omega 2\sqrt{ W(\phi_k) a_{\vb Q_k}(\nabla\phi_k)}\d x \\
    &=2\int_\Omega \sqrt{a_{\vb Q_k}\qty(\sqrt{W(\phi_k)}\nabla\phi_k)} \d x \\
    &= 2\int_\Omega \sqrt{a_{\vb Q_k}(\nabla\Phi(\phi_k))} \d x.
\end{aligned}
\end{equation}
Inspired by the duality argument from \cite{modica_gradient_1987}, we define
\begin{equation} \label{sublin-form}
\begin{aligned}
    V(\bm\mu_1,\bm\mu_2) = \sup\bigg\{&
    \int_\Omega \bm g_1\cdot\bm\mu_1\d x + \int_\Omega \bm g_2\cdot \bm\mu_2\d x:\\
    &\bm g_1,\bm g_2\in C_0(\Omega),\
    \sup_\Omega (|\bm g_1|^2+|\bm g_2|^2)\le 1\bigg\},
\end{aligned}
\end{equation}
where $\bm\mu_i$ are vector-valued Radon measures. $V$ is obviously nonnegative for all $\bm\mu_1,\bm\mu_2$.
(Note that $V$ is defined similarly to the total variation \cite{evans_measure_2015,giusti_minimal_1984}.)
We make the following observations.
\begin{lemma} \label{sublin-form-prop}
\begin{enumerate}[\rm (a)]
    \item $V$ is weak-$*$ lower semi-continuous with respect to Radon measures.
    \item For $\bm\mu_1= \nabla\Phi(\phi_k)$ and $\bm\mu_2=\sqrt{\omega_a/\omega_p}(\vb Q_k+s_+\vb I/3)\nabla\Phi(\phi_k),$ it holds that
    \[V(\bm\mu_1,\bm\mu_2)=\int_\Omega \sqrt{a_{\vb Q_k}(\nabla\Phi(\phi_\ve))}\d x\]
    \item For $\bm\mu_1=\D\Phi(I_D)$ and $\bm\mu_2=\sqrt{\omega_a/\omega_p}(\vb Q+s_+\vb I/3)\D\Phi(I_D)$, it holds that
    \[V(\bm\mu_1,\bm\mu_2)=\frac16\int_{\pt D} a_{\vb Q}(\bm\nu)\d S.\]
\end{enumerate}
\end{lemma}
\begin{proof}[Proof of Lemma \ref{sublin-form-prop}]
(a) Suppose that the sequence $\{\bm\mu_1^{(k)}, \bm\mu_2^{(k)}\}$ weak-$*$ converges to $\{\bm\mu_1,\bm\mu_2\}$. For every pair $\bm g_1(x),\bm g_2(x)\in C_0(\Omega)$ satisfying $|\bm g_1|^2+|\bm g_2|^2\le 1,\forall x$, it holds that
\[ \int_\Omega \bm g_1\cdot\bm\mu_1^{(k)}\d x + \int_\Omega \bm g_2\cdot \bm\mu_2^{(k)}\d x \le V(\bm\mu_1^{(k)},\bm\mu_2^{(k)}).\]
Taking the lower limit as $k\to\infty$ in the inequality above, we get
\[ \int_\Omega \bm g_1\cdot\bm\mu_1 \d x + \int_\Omega \bm g_2\cdot \bm\mu_2 \d x \le \liminf_{k\to\infty} V(\bm\mu_1,\bm\mu_2), \]
where the LHS results from weak-$*$ convergence.
Then, taking the supremum leads to our conclusion.

(b) First,
\begin{equation} \label{V-le-TV}
\begin{aligned}
    V(\bm\mu_1,\bm\mu_2)
    &\le \int_\Omega \sqrt{|\nabla\Phi(\phi_k)|^2 + \frac{\omega_a}{\omega_p} \left| \qty(\vb Q_k+\frac{s_+}{3}\vb I)\nabla\Phi(\phi_k) \right|^2}\d x \\
    &=\int_\Omega \sqrt{a_{\vb Q_k}(\nabla\Phi(\phi_\ve))}\d x.
\end{aligned}
\end{equation}
Recall the definition \eqref{a-Q-xi}.
Moreover, since both $\bm\mu_1$ and $\bm\mu_2$ are $L^1$ functions, we can choose $\bm g_1,\bm g_2$ arbitrarily close to
\[\bm v_1^* = \frac{\nabla\Phi(\phi_k)}{\sqrt{a_{\vb Q_k}(\nabla\Phi(\phi_\ve))}},\ \bm v_2^*=\frac{\sqrt{\omega_a/\omega_p}(\vb Q_k+\frac{s_+}{3}\vb I)\nabla\Phi(\phi_k)}{\sqrt{a_{\vb Q_k}(\nabla\Phi(\phi_\ve))}}\]
almost everywhere.
Note that $|\bm v_1^*|^2+|\bm v_2^*|^2=1$ according to their definition.
Then,
\[\int_\Omega \qty[ \bm g_1\cdot\nabla\Phi(\phi_k) + \bm g_2\cdot \sqrt{\frac{\omega_a}{\omega_p}}\qty(\vb Q_k+\frac{s_+}{3}\vb I)\nabla\Phi(\phi_k)] \d x \]
can be arbitrarily close to the RHS in \eqref{V-le-TV}, so the equality must hold. That is,
\[V(\bm\mu_1,\bm\mu_2)=\int_\Omega \sqrt{a_{\vb Q_k}(\nabla\Phi(\phi_\ve))}\d x\]
as desired.

(c) It is well-known \cite{evans_measure_2015,giusti_minimal_1984} that when $D$ is a smooth region with $\pt D\subset\Omega$, $\D I_D$ is the normal vector measure on $\pt D$, such that
\[\int_\Omega \bm f\cdot \D I_D \d x = -\int_{\pt D} \bm f\cdot\bm\nu \d S,\]
where $\bm\nu$ is the normal vector, and the negative sign comes from integration by parts. Therefore, we can write according to the definition \eqref{sublin-form} that
\begin{equation} \label{V-eq-bndint}
\begin{aligned}
    V(\bm\mu_1,\mu_2)&=\sup_{|\bm g_1|^2+|\bm g_2|^2\le 1}\int_\Omega \qty[ \bm g_1\cdot\D \Phi(I_D) + \bm g_2\cdot \sqrt{\frac{\omega_a}{\omega_p}}\qty(\vb Q+\frac{s_+}{3}\vb I)\D \Phi(I_D)] \d x  \\
    &=\sup_{|\bm g_1|^2+|\bm g_2|^2\le 1}-\frac16\int_{\pt D} \qty[ \bm g_1\cdot\bm\nu + \bm g_2\cdot \sqrt{\frac{\omega_a}{\omega_p}}\qty(\vb Q+\frac{s_+}{3}\vb I)\bm\nu] \d S.
\end{aligned}
\end{equation}
Note that $\Phi(0)=0, \Phi(1)=\int_0^1 s(1-s)\d s=\frac16$ and that $I_D$ is binary-valued, so we have rewritten
\[\Phi(I_D)= \Phi(1) I_D = \frac16 \Phi(I_D).\]
Since $\pt D$ is smooth, $\bm g_1, \bm g_2$ can be arbitrarily close to being parallel to the unit vector field on $\pt D$
\[\bm w_1^* = -\frac{\bm\nu}{\sqrt{a_{\vb Q}(\bm\nu)}},\ \bm w_2^*=-\frac{\sqrt{\omega_a/\omega_p}(\vb Q+\frac{s_+}{3}\vb I)\bm\nu}{\sqrt{a_{\vb Q}(\bm\nu)}},\]
so the supremum in \eqref{V-eq-bndint} equals 
\[V(\bm\mu_1,\bm\mu_2) = \frac16 \int_{\pt D} \sqrt{|\bm\nu|^2 + \frac{\omega_a}{\omega_p} \left| \qty(\vb Q_k+\frac{s_+}{3}\vb I)\bm\nu \right|^2}\d x
=\frac16\int_{\pt D} \sqrt{a_{\vb Q}(\bm\nu)}\d x\]
as desired.
\end{proof}

With Lemma \ref{sublin-form-prop} proven, it is straightforward to utilize \eqref{liminf-ge-1} and estimate that
\begin{align*}
    \liminf_{k\to\infty} &\int_\Omega \qty[\ve a_{\vb Q_k}(\nabla\phi_k) + \ve^{-1}W(\phi_k)] \d x 
    \ge 2\liminf_{k\to\infty} \int_\Omega \sqrt{a_{\vb Q_k}(\nabla\Phi(\phi_k))} \d x \\
    &\overset{(b)}{=} 2 \liminf_{k\to\infty} V\left( \nabla\Phi(\phi_k)\d x, \sqrt{\frac{\omega_a}{\omega_p}}\qty(\vb Q_k+\frac{s_+}{3}\vb I)\nabla\Phi(\phi_k)\d x \right) \\
    &\overset{(a)}{\ge} 2 V\left( \D \Phi(I_D)\d x, \sqrt{\frac{\omega_a}{\omega_p}}\qty(\vb Q+\frac{s_+}{3}\vb I)\D \Phi(I_D) \d x\right) \\
    &\overset{(c)}{=}\frac13 \int_{\pt D} \sqrt{a_{\vb Q}(\bm\nu)}\d S.
\end{align*}
The proof of Claim (III) \eqref{gamlim-3-4-inf} is complete.
\end{proof}

The proof of Claim (IV) involves the discussion of a generalized Van der Waals-Cahn-Hilliard functional \eqref{E-per}.

\begin{lemma} \label{gamma-limsup} \label{pg:chgen}
Work on a bounded Lipschitz domain $\Omega\subset\BbbR^n$ in general dimensions $n$.
Define the generalized Van der Waals-Cahn-Hilliard functional $L_\ve^a: H^1(\Omega) \to \BbbR\cup \{+\infty\}$ as 
\begin{equation} \label{cahn-hill-gen}
    L_\ve^a(\phi) = \int_\Omega \qty[\ve^{-1} W(\phi) + \ve a(x,\nabla\phi)]\d x,
\end{equation}
with the following assumptions.
\begin{enumerate}[\rm(i)]
\item $a(x,\xi) = a^{ij}(x) \xi_i \xi_j$ and repeated indices are summed over, following Einstein's summation convention. The symmetric-matrix-valued function $\vb A=((a^{ij})):\overline{\Omega} \to \BbbR^{n\times n}$ is continuous and satisfies the uniform elliptic condition
\begin{equation} \label{uni-ell}
    \lambda\vb{I} \preceq \vb A(x) \preceq \Lambda \vb I, \quad \forall x.
\end{equation}
The notation $\vb A\preceq\vb B$ means that $\vb B-\vb A$ is positive definite, where $\vb A,\vb B$ are symmetric matrices.
\item $W(s)$ is a nonnegative continuous double-well potential with strict minima at $s=\alpha,\beta$ ($\alpha<\beta$), and $W(\alpha)=W(\beta)=0$.
\item $D\subset\subset\Omega$ is a region with $C^2$ boundary.
\end{enumerate}
Let $\phi_0=\alpha I_D + \beta I_{D^c}$ be the separation of the phases $\alpha$ and $\beta$.
Also, denote the perimeter of $D$ by $\mcalP(D)=\mathcal{H}_{n-1}(\pt D)$ (the $(n-1)$-dimensional Hausdorff measure \cite{evans_measure_2015}) and introduce the constant
\[c_0=\int_\alpha^\beta \sqrt{W(s)}\d s.\]
Suppose further that $\phi_0=\alpha I_D+\beta I_{D^c}$ satisfies the volume constraint
\begin{equation} \label{vol-con-1}
    \int_\Omega\phi\d x = m
\end{equation}
where $\alpha|\Omega|<m<\beta|\Omega|$, or equivalently $|D|=\frac{\beta|\Omega|-m}{\beta-\alpha}.$
Then, there exists a sequence $\phi_\ve \in W^{1,\infty}(\Omega)$ satisfying the constraint \eqref{vol-con-1}, $\alpha\le\phi_\ve\le \beta$, $\phi_\ve\to \phi_0$ a.e., such that
\begin{equation}
    \limsup_{\ve\to 0} L_\ve^a(\phi_\ve) \le 2c_0 \int_{\pt D} |a(x,\bm\nu)|^{\frac12}\d S(x).
\end{equation}
\end{lemma}
\begin{remark} \label{pg:chgenrm}
We make a few remarks on Lemma \ref{gamma-limsup}.
\begin{itemize}
\item The lemma is stated in exactly the same fashion as in \cite[Prop. 2]{modica_gradient_1987}, where the same properties are derived for the degenerate case $a^{ij}(x,p)\equiv |p|^2$ and \eqref{cahn-hill-gen} is reduced to the classical Van der Waals-Cahn-Hilliard energy. 
\item Similar discussions on generalized Van der Waals-Cahn-Hilliard energy functionals can be found in e.g. \cite{ansini_gradient_2003,bouchitte_singular_1990}, where the $\Gamma$-limit of functionals in the form of singular perturbation on the gradient is studied in detail.
\item The theory of Lemma \ref{gamma-limsup} can be applied to our model with $\alpha=1,\beta=0,W(s)=s^2(1-s)^2$, $a(x,\phi_\ve)=a_{\vb Q}(\phi_\ve)$ and $\phi_0=I_D$. When $\alpha>\beta$, one can set $\tilde W(s)=W(-s)$ and apply the theory to $\{-\phi_\ve\}$, with $\tilde W$ substituted for $W$.
\end{itemize}
\end{remark}

We present a constructive proof of Lemma \ref{gamma-limsup}.
A useful insight into the constuction of the sequence $\{\phi_\ve\}$ (also named the \textit{recovery sequence}) is the property of \textit{equipartition of energy} in energy functionals of Van der Waals-Cahn-Hilliard form \cite{dai_convergence_2018}.
It states that if the upper bound is to hold, then it must hold a.e. that
\[\ve^{-1}W(\phi_\ve) \approx \ve a(x,\nabla\phi_\ve),\]
i.e. the double-well potential and the gradient term contribute approximately equally to the overall energy $L^a_\ve$ \eqref{cahn-hill-gen}.
The equipartition property motivates us to construct the recovery sequence.

\begin{proof}[Proof of Lemma \ref{gamma-limsup}]

\textbf{Step 1. Generalized signed distance function.} \label{pg:chgen-pf}
First, we construct (locally) a Lipschitz function $h^a$ satisfying $a(x,\nabla h^a)=1,$ a.e..
$h^a$ is essentially the solution to the following first-order nonlinear partial differential equation (PDE)
\begin{equation} \label{eik-gen}
    \begin{cases}
        a^{ij}(x) u_{x_i}u_{x_j} = 1, & x\in\Omega\\
        u=0, & x\in \pt D,
    \end{cases}
\end{equation}
which can be solved with the method of characteristics. The solution is a generalized signed distance function (SDF).

Assume for the moment that $a^{ij}(x)\in C^\infty(\overline{\Omega})$. Then by \cite[sect. 3.2]{evans_partial_2010} (with $F(x,u,p)=a(x,p)-1$), the solution to \eqref{eik-gen} is given by the \textit{characteristic curves} $(\bm X(s), U(s), \bm P(s))$ determined by the dynamical system
\begin{equation} \label{eik-moc}
    \begin{cases}
        \dot X_i = -\pt_{p_i} F(\bm X,U,\bm P)=2a^{ij}(\bm X)P_j, \\
        \dot U =\nabla_p F(\bm X,U,\bm P)\cdot\bm P= 2a^{ij}(\bm X)P_i P_j, \\
        \dot P_i =-\pt_{x_i} F(\bm X,U,\bm P)-P_i \pt_uF(\bm X,U,\bm P) = - a^{jk}_{x_i}(\bm X) P_j P_k,
    \end{cases}
    \quad \qty(\dot{(*)} = \frac{\pt(*)}{\pt s})
\end{equation}
(Einstein's summation convention is used) with initial values
\begin{equation} \label{eik-moc-ic}
    \begin{cases}
        \bm X(0)=y\in \pt D, \\
        U(0)=0, \\
        \bm P(0)=|a(y,\bm\nu)|^{-\frac12}\bm\nu.
    \end{cases}
\end{equation}
The curve $\bm X(s)$ originates from the surface of boundary conditions $\pt D$, and the functions $U(s),\bm P(s)$ portray respectively the evolution of $h^a(x)$ and $\nabla h^a(x)$ along the curves.
The initial values \eqref{eik-moc-ic} are designed to be compatible with the boundary condition $h^a=0$ and $a(x,\nabla h^a)=1$.


With the theory of first-order nonlinear PDE's, we assert that the system \eqref{eik-moc} solves \eqref{eik-gen} correctly in a neighbourhood $V\supset \pt D$, with $U(s), \bm P(s)$ standing for the respective value of $h^a, \nabla h^a$ at the point $\bm X(s)$.
\begin{lemma} \label{moc-solves-eik}
Assume that $\pt D$ is compact and $C^2$ smooth.
Then, there exists an open interval $I$ containing 0 such that: 
\begin{enumerate}[\rm(a)]
    \item For each $y\in \pt D$ and $s\in I$, a unique $C^1$ characteristic curve $(\bm X(s),U(s),\bm P(s))$ exists with the initial conditions \eqref{eik-moc-ic}.
    \item The mapping
    \begin{align*}
        I\times\pt D&\to \Omega\times \BbbR \times \BbbR^n \\
        (s,y)&\mapsto(\bm X,U,\bm P)
    \end{align*}
    is $C^1$ with respect to $s\in I$ and $y\in\pt D$.
    \item The mapping $(s,y)\mapsto \bm X(s;y)$ for $s\in I,y\in\pt D$ is bijective, and its inverse is also $C^1$, i.e. there exists $C^1$ inverse mappings $s=s(x),y=y(x)$.
    \item On the open region containing $D$ defined by
    \begin{equation} \label{moc-domain}
        V=\left\{\bm X(s;y): s\in I, y\in \pt D\right\},
    \end{equation}
    the function
    \begin{equation} \label{gsdf}
        h^a(x) = U(s(x); y(x))
    \end{equation}
    is a $C^1$ classical solution to the equation \eqref{eik-gen}, with $\nabla h^a(x)=\bm P(s(x);y(x))$.
\end{enumerate}
\end{lemma}

We can also make the following observations on the solution $h^a(x)$ \eqref{gsdf} restricted to $V$ defined by \eqref{moc-domain}.
It states that the sign of $h^a(x)$ determines which side of $\pt D$ the point $x$ is on, a property shared by the classical SDF.
\begin{lemma} \label{eik-moc-prop}
Assume the same as Lemma \ref{moc-solves-eik}.
Then for any $x\in V$, the position of $x$ relative to $D$ is determined by the sign of $h^a(x)$. Specifically,
\[\begin{cases}
    x\in D, & h^a(x)<0, \\
    x\in\pt D, & h^a(x)=0, \\
    x\in (\overline{D})^c, &h^a(x)>0.
\end{cases}\]
\end{lemma}
\begin{proof}
Since $U(s)=h^a(\bm X(s)), \bm P(s)=\nabla h^a(\bm X(s))$ solves \eqref{eik-gen} by Lemma \ref{moc-solves-eik},
\[a(\bm X(s),\bm P(s))\equiv 1\]
for all $s\in I$.
Therefore, we notice from the characteristic equation \eqref{eik-moc} that $\dot U(s)=2a(\bm X,\bm P)=2,$
which indicates 
\begin{equation} \label{U-eq-2s}
    h^a(x)=U(s(x))=2s.
\end{equation}

By \eqref{eik-moc-ic}, the initial values satisfy $\dot{\bm X}(0)=2a^{ij}(y)P_j(0), \bm P(0)=c \bm\nu$ ($c>0$), so
\[\dot{\bm X}(0)\cdot\bm\nu =c 2a^{ij}(y) \nu_i\nu_j>0 \]
since $a^{ij}$ satisfies the elliptic condition \eqref{uni-ell}. As $\bm\nu$ is the exterior normal vector, the characteristic curve points outside the region $D$ as $s$ increases. There is also no second crossing of the curve $\bm X(s)$ with $\pt D$ when $s\neq 0$ because the solution is unique.
Therefore, $s>0$ corresponds to the outer side of $D$ and $s<0$ corresponds to the inner side.
By \eqref{U-eq-2s}, the signs of $h^a$ and $s(x)$ are identical, so we yield Lemma \ref{eik-moc-prop}.
\end{proof}

The following corollary characterizes level sets of $h^a$, which are local isomorphisms of the boundary $\pt D$.
\begin{corollary} \label{gsdf-level-set}
Use the same assumptions and notations as Lemma \ref{moc-solves-eik}.
The level sets $\Sigma_t=\{h^a=t\}$ of $h^a(x)$ satisfies the following properties.
\begin{enumerate}[\rm(a)]
    \item For all $t$ with $\frac{t}{2}\in I$, $\Sigma_t$ is a $C^1$ hypersurface.
    \item For any continuous function $g(x):\Omega\to \BbbR$, the integral
    \[\int_{\Sigma_t} g(x)\d S(x)\]
    is continuous with respect to $t$. Specifically, when $g(x)\equiv 1$, we get the continuity of $\mathcal{H}_{n-1}(\Sigma_t)$ with respect to $t$.
\end{enumerate}
\end{corollary}
\begin{proof}
Recall from \eqref{U-eq-2s} that $U(s)=2s$, so the level set $h^a=t$ is expressed equivalently as
\[h^a(x)=t \Leftrightarrow s(x)=\frac{t}{2}.\]
That is, as long as $\frac{t}{2}\in I$, $\Sigma_t$ is a section of the $C^1$ bijective mapping $(s,y)\mapsto \bm X(s;y)$ (by Lemma \ref{moc-solves-eik}(c)) at $s=\frac{t}{2}$:
\begin{equation} \label{Sigma-t}
    \Sigma_t = \{\bm X(t/2;y): y\in\pt D\}.
\end{equation}
Therefore, $s=\frac{t}{2}$ is a $C^1$ parametrization of the surface family $\{\Sigma_t\}_{t}$, and the assertions of Corollary \ref{gsdf-level-set} follow immediately.
\end{proof}


\textbf{Step 2. Approximate phase transition sequence.}
Next, we construct the approximate phase transition sequence $\{\phi_\ve\}$ over $\Omega$, such that $\phi_\ve\to \phi_0$ a.e. as $\ve\to 0$.

We propose a function $\chi_\ve:\BbbR\to\BbbR$ as the solution to the following differential equation.
\begin{equation} \label{chi-eq}
    \begin{cases}
        \ve \chi_\ve'(t) = \sqrt{W(\chi_\ve(t))+\ve}, \\
        \chi_\ve(0)=\alpha.
    \end{cases}
\end{equation}
The function $\chi_\ve$, taken from \cite{modica_gradient_1987}, is intended to approximate the standing-wave solution $q(t)$ as $\ve\to 0$, which satisfies that $ q(-\infty)=\alpha, q(+\infty)=\beta$ and minimizes the one-dimensional Van der Waals-Cahn-Hilliard energy \cite{evans_phase_1992,yue_diffuse-interface_2004}.

The following properties of $\chi_\ve$ are evident.
\begin{lemma} \label{chi-prop}
The function $\chi_\ve$ satisfies that:
\begin{enumerate}[\rm(a)]
    \item $\chi_\ve$ is strictly increasing.
    \item There exists a number $\eta_\ve>0$ such that $\chi_\ve(\eta_\ve)=\beta$, i.e. $\chi_\ve$ changes from $\alpha$ to $\beta$ over the interval $[0,\eta_\ve]$.
    \item $\eta_\ve\to 0$ as $\ve\to 0$.
\end{enumerate}
\end{lemma}

Since we are only interested in the phase transition from $\alpha$ to $\beta$, we truncate the function $\chi_\ve$ to the interval $[\alpha,\beta]$:
\begin{equation} \label{chi-trunc}
    \tilde\chi_\ve(t)=\min\left\{\beta, \max\{\alpha, \chi_\ve(t)\}\right\}=\begin{cases}
        \alpha, & t<0, \\
        \chi_\ve(t), & t\in[0,\eta_\ve], \\
        \beta, & t>\eta_\ve
    \end{cases}.
\end{equation}
$\tilde\chi_\ve$ is an increasing and Lipschitz continuous function $\BbbR\to[\alpha,\beta]$.
We also denote the sharp phase transition from $\alpha$ to $\beta$ by
\begin{equation} \label{chi0}
    \chi_0(t) = \begin{cases}
        \alpha, & t<0 \\ \beta, & t>0
    \end{cases}.
\end{equation}
A sketch of their graphs can be found in Figure \ref{chi-sketch}.
\begin{figure}[th]
    \centering
\begin{tikzpicture}[x=0.7cm,y=0.7cm]
    \draw[->,color=black!60!] (-3.5,0) -- (3.5,0) node[anchor=west]{$t$}; 
    \draw[->,color=black!60!] (0,-.5) -- (0,3); 
    \draw[color=black] (-3,0) -- (0,0) ..controls(.4,.3)and(.4,1.7).. (.8,2) -- (3,2) node[anchor=north west]{$\tilde\chi_\ve(t)$};
    \draw[very thick, dashed,color=red] (-2.5,0) -- (0,0) -- (0,2)-- (2.5,2) node[anchor=south west]{$\chi_0(t)$};
    \draw[dashed] (.8,2) -- (.8,0) node[anchor=north]{\small$\eta_\ve$};
    \draw[dashed] (0,0) node[anchor=north east]{\small$0$};
\end{tikzpicture}
\caption{Sketch of $\tilde\chi_\ve(t)$ compared with $\chi_0$} \label{chi-sketch}
\end{figure}


For each $\ve$ satisfying $\eta_\ve<\eta_*$, we define the phase transition sequence on $V$:
\begin{equation} \label{phi-eps}
    \phi_\ve(x) = \tilde\chi_\ve(h^a(x)+\delta_\ve)=\begin{cases}
        \chi_\ve(h^a(x)+\delta_\ve), & -\delta_\ve\le h^a(x)\le \eta_\ve-\delta_\ve, \\
        \alpha, & h^a(x)< -\delta_\ve, \\
        \beta, & h^a(x)> \eta_\ve-\delta_\ve,
    \end{cases}
\end{equation}
where $h^a$ is the generalized SDF defined by \eqref{gsdf}, $\tilde\chi_\ve$ is defined in \eqref{chi-trunc}, and $\delta_\ve\in[0,\eta_\ve]$ is yet to be determined.
The sharp phase transition $\phi_0= \alpha I_D+\beta I_{D^c}$ can also be expressed as
\begin{equation} \label{phi-0}
    \phi_0(x)=\chi_0(h^a(x))
\end{equation}
over $V$, where $\chi_0$ is defined in \eqref{chi0}, because the sign of $h^a$ in $V$ is determined by whether $x\in D$ or $x\notin D$.

We choose $\delta_\ve\in[0,\eta_\ve]$ to meet the volume constraint \eqref{vol-con-1} restricted to $V$.
\begin{lemma} \label{delta-exst}
There exists $\delta_\ve\in[0,\eta_\ve]$ such that 
\[\int_V \phi_\ve(x)\d x = \int_{V} \phi_0(x)\d x.\]
\end{lemma}
\begin{proof}
Evidently,
\[\tilde\chi_\ve(t)\le \chi_0(t)\le \tilde\chi_\ve(t+\eta_\ve),\ \forall t\in\BbbR,\]
as visible from Figure \ref{chi-sketch}, so substituting $h^a(x)$ for $t$ and integrating gives us
\[\int_V \tilde\chi_\ve(h^a(x))\d x \le \int_V \chi_0(h^a(x))\d x=m \le \int_V\chi_\ve(h^a(x)+\eta_\ve)\d x,\]
where $m=\int_V \phi_0\d x$ is from the volume constraint \eqref{vol-con-1}.
Since the integral
\[\int_V \tilde\chi_\ve(h^a(x)+\delta)\d x\]
is continuous and monotonic with respect to $\delta$,
there exists $\delta_\ve \in [0,\eta_\ve]$ such that $\int_V \tilde\chi_\ve(h^a+\delta_\ve)\d x=m$ by the mean value theorem.
\end{proof}

We then extend $\phi_\ve$ to the entire $\Omega$.
\begin{lemma} \label{ext-phi}
Define
\begin{equation} \label{phi-eps-ext}
    \phi_\ve(x)=\begin{cases}
        \alpha, & x\in D\cap V^c, \\
        \beta, & x\in D^c\cap V^c,
    \end{cases}
\end{equation}
and then:
\begin{enumerate}[\rm(a)]
\item The extension \eqref{phi-eps-ext} is a Lipschitz function over $\Omega$.
\item $\phi_\ve$ satisfies the volume constraint \eqref{vol-con}.
\item $\phi_\ve\to \phi_0$ a.e., where $\phi_0=\alpha I_D+\beta I_{D^c}$.
\end{enumerate}
\end{lemma}
\begin{proof}
(a) We need to prove the differentiability near the boundary $\pt V$.
Since $V$ is defined by images of the $C^1$ mapping $\bm X(s;y)$ with $s\in (-\eta_*,\eta_*)$, the boundary $\pt V$ are two patches of level set surfaces $\{s=\pm\eta_*\}$ (which are the level sets $\Sigma_{\pm 2\eta_*}=\{h^a=\pm 2\eta_*\}$ defined in \eqref{Sigma-t}).
On the positive patch $s=\eta_*$, $h^a(x)=2\eta_*>\eta_\ve\ge \eta_\ve-\delta_\ve$ by the assumption that $\eta_*>\eta_\ve$, so by \eqref{phi-eps} we have that $\phi_\ve|_{U\cap V}\equiv \beta$ for a small neighbourhood $U$ containing $x\in \Sigma_{2\eta_*}$ restricted to $V$; moreover, the value of $\phi_\ve$ on $U\cap V^c$ is also constantly $\beta$ by definition \eqref{phi-eps-ext}. Hence, the extended $\phi_\ve$ is differentiable (constant) near $x.$
The same argument applies to the negative patch $s=-\eta_*$.

(b) We notice that the values of $\phi_\ve$ and $\phi_0$ agree on $V^c$ by \eqref{phi-eps-ext} and \eqref{phi-0}, so $\int_{V^c} \phi_\ve(x)\d x = \int_{V^c} \phi_0(x)\d x.$
Meanwhile, Lemma \ref{delta-exst} states that $\int_V \phi_\ve(x)\d x=\int_V \phi_0(x)\d x.$
Adding them up gives us the volume constraint \eqref{vol-con} on the entire $\Omega$.

(c) It suffices to prove that $\phi_\ve\to \phi_0$ ($\ve\to 0$) whenever $x\notin \pt D$.
If $x\notin V$, then $\phi_\ve(x)\equiv \phi_0(x)$ by definition \eqref{phi-eps-ext}.
If $x\in V$ and $x\notin \pt D$, then $h^a(x)\neq 0$ by the uniqueness of characteristic curves.
Since $\eta_\ve\to 0$ and $0\le\delta_\ve\le \eta_\ve$ (by Lemma \ref{chi-prop} and Lemma \ref{delta-exst}), $h^a\notin(-\delta_\ve, \eta_\ve-\delta_\ve)$ when $\ve\to 0$. Therefore $\phi_\ve(x)=\phi_0(x)$ for sufficiently small $\ve$, and hence the convergence.
\end{proof}

\textbf{Step 3. Integration with coarea formula.}
Finally, we evaluate the functional \eqref{cahn-hill-gen} with the aid of the sequence $\phi_\ve$.
We utilize the following coarea formula for a Lipschitz function $h\in W^{1,\infty}(\Omega)$ and an integrable function $f\in L^1(\Omega)$ \cite{evans_partial_2010,evans_measure_2015,giusti_minimal_1984}, which translates volume integrals to surface integrals.
\begin{equation} \label{coarea}
    \int_{\Omega} f(x) |\nabla h|\d x=\int_{-\infty}^\infty \d r \int_{\{h=r\}\cap \Omega} f(x)\d S(x).
\end{equation}
The steps are identical to those of Modica \cite{modica_gradient_1987}.

By \eqref{phi-eps}, when $h^a\ge\eta_\ve-\delta_\ve$ or $h^a\le -\delta_\ve$, $\phi_\ve$ is constantly $\alpha$ or $\beta$, so both $W(\phi_\ve)$ and $\nabla\phi_\ve$ are zero.
Therefore, the integral \eqref{cahn-hill-gen} is supported on the ``strip'' $\{-\delta_\ve < h^a < \eta_\ve-\delta_\ve\},$ or equivalently, $\{\alpha<\phi_\ve<\beta\}.$
Using the differential equation \eqref{chi-eq} and the chain rule, we get
\[ \ve a(x,\nabla\phi_\ve) = \ve^{-1}(W(\phi_\ve)+\ve) a(x,\nabla h^a)=\ve^{-1}(W(\phi_\ve)+\ve).\]
Thus, we transcribe the integral into
\begin{align*}
    L_\ve^a(\phi_\ve)&=\int_{\{\alpha<\phi_\ve<\beta\}} \ve^{-1}(2W(\phi_\ve)+\ve)\d x \\
    &=\int_{-\delta_\ve}^{\eta_\ve-\delta_\ve} \d t \int_{\Sigma_t} \frac{\ve^{-1}}{|\nabla h^a|} (2W(\phi_\ve)+\ve) \d S(x),
\end{align*}
where the coarea formula \eqref{coarea} is applied to $h^a$.
Note that because of uniform ellipticity \eqref{uni-ell},
\[1=|a(x,\nabla h^a)|\le \Lambda |\nabla h^a|^2 \Rightarrow |\nabla h^a| \ge \Lambda^{-\frac12},\]
and thus $|\nabla h^a|$ can be placed on the denominator.
Substitute the definition \eqref{phi-eps} in the expression above, and extract the factor $W(\phi_\ve)+\ve$ from the surface integral since it is constant over that surface.
\begin{align*}
    L_\ve^a(\phi_\ve)
    &= \int_{-\delta_\ve}^{\eta_\ve-\delta_\ve} \d t \int_{\Sigma_t} \frac{2W(\tilde\chi_\ve(t+\delta_\ve))+\ve}{\ve|\nabla h^a|} \d S(x) \\
    &= \int_{0}^{\eta_\ve} \frac{2W(\chi_\ve(t))+\ve}{\ve} \d t \int_{\Sigma_{t-\delta_\ve}} \frac{1}{|\nabla h^a|} \d S(x).
\end{align*}
We have shifted the interval of $t$ by $\delta_\ve$.
As $\chi_\ve(t)$ is strictly increasing by Lemma \ref{chi-prop}(a), we make the change of variable $r=\chi_\ve(t)\in(\alpha,\beta)$, whose substitution formula is $\d r=\chi_\ve'(t)\d t=\ve^{-1}\sqrt{W(r)+\ve}\d t.$ Then,
\begin{equation} \label{pinkie}
    L_\ve^a(\phi_\ve) = \int_\alpha^\beta \frac{2W(r)+\ve}{\sqrt{W(r)+\ve}}\d r \int_{\Sigma_{\chi_\ve^{-1}(r)-\delta_\ve}}\frac{1}{|\nabla h^a|}\d S(x),
\end{equation}
where $\chi_\ve^{-1}(r): (\alpha,\beta)\to(0, \eta_\ve)$ is the inverse function of $\chi_\ve$.
By Lemma \ref{chi-prop}, $\eta_\ve\to 0$ and $\delta_\ve\to 0$ as $\ve\to 0$, so $\chi_\ve^{-1}(r)-\delta_\ve\to 0$ uniformly for all $r\in(\alpha,\beta)$. 
$\frac{1}{|\nabla h^a|}$ is also continuous over $V$ because $h^a\in C^1$.
Thus, by Lemma \ref{gsdf-level-set}(b) we can assert that
\[\lim_{\ve\to 0} \int_{\Sigma_{\chi_\ve^{-1}(r)-\delta_\ve}}\frac{1}{|\nabla h^a|}\d S(x) = \int_{\Sigma_0} \frac{1}{|\nabla h^a|}\d S(x)\]
uniformly for all $r\in(\alpha,\beta)$.
Moreover, $\frac{2W(r)+\ve}{\sqrt{W(r)+\ve}} \le 2\sqrt{W(r)+\ve}$ is uniformly bounded above by a continuous function for all $r\in[\alpha,\beta]$.
Therefore, we apply the dominated convergence theorem on \eqref{pinkie} by taking the limit $\ve\to 0$ under the integral sign, and get
\[\lim_{\ve\to 0} L_\ve^a(\phi_\ve)
= \int_\alpha^\beta 2\sqrt{W(r)}\d r \int_{\Sigma_0} \frac{1}{|\nabla h^a|}\d S(x). \]
Since $\Sigma_0$ is just $\pt D$, and since $|\nabla h^a(x)|=|\bm P(0;x)|=|a(x,\bm\nu)|^{-\frac12}$ ($x\in\pt D$) according to \eqref{eik-moc-ic},
\[\lim_{\ve\to 0} L_\ve^a(\phi_\ve) = 2c_0 \int_{\pt D} |a(x,\bm\nu)|^{\frac12}\d S(x),\]
matching the form in Lemma \ref{gamma-limsup}.

\textbf{Step 4. An approximation argument.}
All previous derivations are based on the assumption that $a^{ij}(x)\in C^\infty$.
For the general case $a^{ij}\in C^0$, we choose a sequence of smooth matrix functions $\vb A_k(x)=((a_k^{ij}(x)))\in C^\infty$ such that $\vb A_k\to \vb A$ uniformly and that $\vb A_k(x)\succeq \vb A(x),\forall x$. Denote the quadratic form associated with $\vb A_k$ by $a_k(x,\xi)=a_k^{ij}(x)\xi_i\xi_j.$
W.l.o.g., assume further that the smooth approximations $\{\vb A_k\}$ satisfy the uniform ellipticity conditions \eqref{uni-ell} with the same constants $\lambda,\Lambda$.

Before proceeding with the argument, we need to further characterize the behaviour of $h^a$.
\begin{lemma} \label{gsdf-ge-dist}
Under the assumptions of Lemma \ref{moc-solves-eik}, the solution $h^a(x)=U(s(x);y(x))$ satisfies that
\[\Lambda^{-1/2} \dist(x,\pt D) \le |h^a(x)| \le \lambda^{-1/2} \dist(x,\pt D),\quad \forall x\in V,\]
where $\Lambda$ is the constant from \eqref{uni-ell} and $\dist$ is the distance function.
\end{lemma}
\begin{proof}
By \eqref{uni-ell} and the fact that $a(x,\nabla h^a)=1$, we find that
\[ \Lambda^{-1/2}\le|\nabla h^a|\le \lambda^{-1/2}.\]
For any line segment in $V$ connecting $y\in \pt D$ and $x$,
\[|h^a(x)|=\left|\int_y^x \nabla h^a\cdot\d\bm l \right| \le \lambda^{-1/2} \int_y^x |\d\bm l|=\lambda^{-1/2}|x-y|.\]
by the triangular inequality of line integrals. Taking the infimum over $y\in \pt D$ leads to the right inequality.

For each $x\in V$, denote $s_0=s(x)$ the parameter along the characteristic curve and $y=y(x)$ the initial point of the curve. By the equation \eqref{eik-moc} we get
\[|\dot{\bm X}|=2|\vb A\bm P|\le 2\Lambda^{1/2}\sqrt{a^{ij}P_iP_j}=2\Lambda^{1/2}, \quad\forall s,\]
since $\Lambda$ bounds the largest eigenvalue of $\vb A$.
W.l.o.g. letting $s_0>0$, we integrate along the characteristic curve from $y$ to $x$ to get 
\[\dist(x,\pt D) \le |x-y| \le \int_0^{s_0} |\dot{\bm X}|\d s \le 2\Lambda^{1/2} s.\]
Moreover, by \eqref{U-eq-2s} we have $h^a(x)=U(s)=2s$, and the LHS inequality follows immediately.
\end{proof}

Based on the matrix function $\vb A_k$, we construct the generalized SDF $h^{a_k}(x)$ and the phase transition sequence $\phi^k_\ve(x)$ with the same method as in \eqref{gsdf} and \eqref{phi-eps-ext}, respectively.
Note that the function $\chi_\ve$ and the number $\eta_\ve$, as stated in Lemma \ref{chi-prop}, are independent of $k$. 
We find the following fact important.
\begin{lemma} \label{phikeps-uniform}
For all $x\notin \pt D$, $\phi^k_\ve(x) \to \phi_0(x)$ as $\ve\to 0$, and the convergence is \emph{uniform in $k$.}
\end{lemma}
\begin{proof}
Since the uniform elliptic constants $\lambda,\Lambda$ in \eqref{uni-ell} are the same for all $\{\vb A_k\}$ by assumption, the conclusion of Lemma \ref{gsdf-ge-dist} holds. Thus, $x\notin \pt D$, i.e. $\dist (x,\pt D)>0$ implies that
\[|h^{a_k}(x)|\ge \Lambda^{-1/2}\dist(x,\pt D)>0.\]
Then, it holds for all $k$ that $|h^{a_k}(x)|>\eta_\ve$ whenever $\eta_\ve< \Lambda^{-1/2}\dist(x,\pt D)$.
Therefore, the same argument in Lemma \ref{ext-phi}(c) applies, leading to $\phi^k_\ve(x)=\phi_0(x)$ for all $k$ and $\ve$ satisfying $\eta_\ve<\Lambda^{-1/2}\dist(x,\pt D)$. Hence the convergence $\phi^k_\ve(x)\to \phi_0(x)$ is uniform over $k$.
\end{proof}

Denote the generalized Van der Waals-Cahn-Hilliard functional associated with $\vb A_k$ by $L^{a_k}_\ve(\phi)$. For each $k>1$, we iteratively choose $\ve_k<\min\{k^{-1},\ve_{k-1}\}$ such that
\[\sup_{\ve_k\le \ve<\ve_{k-1}} L^{a_k}_{\ve}[\phi^k_\ve] < \int_{\pt D} |a_k(x,\bm\nu)|^{\frac12}\d S(x)+\frac1k,\]
which is possible because Lemma \ref{gamma-limsup} applies to $\vb A_k$.
We construct the sequence $\{\phi_\ve\}$ by defining
\[\phi_\ve = \phi^k_\ve, \quad \ve_{k}\le \ve<\ve_{k-1}.\]
Then by Lemma \ref{phikeps-uniform}, the sequence $\{\phi_\ve\}$ converges to $\phi_0$ for all $x\notin\pt D$. It satisfies the volume constraint \eqref{vol-con-1} since $\{\phi^k_\ve\}$ satisfies the same constraint. Moreover, we compute that
\begin{align*}
    \limsup_{\ve\to 0} L^a_\ve(\phi_\ve)
&\le \limsup_{k\to\infty} \sup_{\ve_k\le \ve<\ve_{k-1}} L^{a_k}_{\ve}[\phi^k_\ve] \\
&\le \limsup_{k\to\infty} \left[ \int_{\pt D} |a_k(x,\bm\nu)|^{\frac12}\d S(x)+\frac1k \right] \\
&=\int_{\pt D} |a(x,\bm\nu)|^{\frac12}\d S(x),
\end{align*}
as desired. In the computation above, the first inequality is based on the assumption that $\vb A\preceq \vb A_k$.
\end{proof}

We proceed to proving Claim (IV).

\begin{proof}[Proof of Claim (IV)]
Let $\phi_\ve$ be the sequence guaranteed by Lemma \ref{gamma-limsup}, with $a(x,p)=a_{\vb Q(x)}(p)$ and $\alpha=1,\beta=0$, which satisfies $\phi_\ve\to I_D$ a.e., $0\le\phi\le 1$ and the volume constraint \eqref{vol-con}.
Then, we apply the lemma to get
\begin{align*}
    \text{LHS of \eqref{gamlim-3-4-sup}}
    &=\limsup_{\ve\to 0} \int_\Omega \qty[ \ve a_{\vb Q}(\nabla\phi_\ve) + \ve^{-1} W(\phi_\ve)]\d x \\
    &\le 2 c_0 \int_{\pt D} |a_{\vb Q}(\bm\nu)|^{\frac12}\d S = \text{RHS of \eqref{gamlim-3-4-sup}}
\end{align*}
as desired, since $c_0=\frac16$.
\end{proof}

Having established the claims, we can prove the lower and upper limits in Proposition \ref{ldgphf-gam-lim}.
\begin{proof}[Proof of Proposition \ref{ldgphf-gam-lim}]
\textbf{i. Lower bound proof.} For the lower bound \eqref{ldgphf-liminf}, we add \eqref{gamlim-1-2-inf} and \eqref{gamlim-3-4-inf} for all sequences $\phi_\ve\to I_D$ in $L^1$ and $\vb Q_\ve\rightharpoonup \vb Q$ in $H_0^1$. As $H_0^1$ is compactly embedded into $L^1$, $\vb Q_\ve\to \vb Q$ in $L^1$ as well.
Using Claims (I) and (III), we perform the following computations to obtain \eqref{ldgphf-liminf}.
\begin{align*}
    \liminf_{\ve\to 0} E_\ve[\vb Q_\ve,\phi_\ve]
    &\ge \liminf_{\ve\to 0} \qty( E^{\rm LdG}[\vb Q_\ve] +\omega_v\lambda^2  E^{\rm void}[\vb Q_\ve,\phi_\ve]) \\
    &\quad + \omega_p\lambda \liminf_{\ve\to 0} \int_\Omega \qty[\ve a_{\vb Q_\ve}(\nabla\phi_\ve) + \ve^{-1}W(\phi_\ve)] \d x \\
    &\ge E^{\rm LdG}[\vb Q]
    +\omega_v\lambda^2 E^{\rm void}[\vb Q,I_D]
    +\frac{\omega_p\lambda}{3} \int_{\pt D} |a_{\vb Q}(\bm\nu)|^{\frac12}\d S(x) \\
    &=E_0[\vb Q,D].
\end{align*}

\textbf{ii. Upper bound proof.} Take the sequence guaranteed by Claim (IV), where $\phi_\ve\to I_D$ a.e. with $0\le\phi_\ve\le 1$ (so $\{\phi_\ve\}$ also satisfies Claim (II)) and the volume constraint \eqref{vol-con} holds. 
Therefore, \eqref{gamlim-1-2-sup} and \eqref{gamlim-3-4-sup} both hold for the sequence $\{\phi_\ve\}$. Adding the two inequalities results in  \eqref{ldgphf-limsup}.
\begin{align*}
    \limsup_{\ve\to 0} E_\ve[\vb Q,\phi_\ve]
    &= \lim_{\ve\to 0} \qty( E^{\rm LdG}[\vb Q] + \omega_v\lambda^2 E^{\rm void}[\vb Q,\phi_\ve]) \\
    &\quad + \omega_p\lambda\limsup_{\ve\to 0} \int_\Omega \qty[\ve a_{\vb Q}(\nabla\phi_\ve) + \ve^{-1}W(\phi_\ve)] \d x \\
    &\le E^{\rm LdG}[\vb Q]
    +\omega_v\lambda^2 E^{\rm void}[\vb Q,I_D]
    +\frac{\omega_p\lambda}{3} \int_{\pt D} |a_{\vb Q}(\bm\nu)|^{\frac12}\d S(x) \\
    &=E_0[\vb Q,D].
\end{align*}
We have proven the assertions of Proposition \ref{ldgphf-gam-lim} as desired.
\end{proof}

Finally, the proof of Theorem \ref{sharp-intf-limit} follows from Proposition \ref{ldgphf-gam-lim} by a standard argument.

\begin{proof}[Proof of Theorem \ref{sharp-intf-limit}]
For any $(\vb Q,D)$ satisfying $\vb Q\in C^0\cap H_0^1$ and $D\subset\subset\Omega$, take the sequence $\{\phi_\ve\}$ guaranteed by Proposition \ref{ldgphf-gam-lim}(b) to get
\begin{align*} 
    E_0[\vb Q,D] &\ge \limsup_{\ve\to 0} E_\ve[\vb Q,\phi_\ve] \\
    &\ge \liminf_{\ve\to 0} E_\ve[\vb Q_\ve^*, \phi_\ve^*]\ge E_0[\vb Q_0,D_0].
\end{align*}
The second inequality comes from the assumption that $(\vb Q_\ve^*,\phi_\ve^*)$ is a minimizer of $E_\ve$, and the third inequality utilizes Proposition \ref{ldgphf-gam-lim}(a) at $(\vb Q_0,D_0)$.
Hence, by definition $(\vb Q_0,D_0)$ minimizes $E_0$ amongst all $(\vb Q,D)$ of interest. 
\end{proof}

\section{Numerical results} \label{sec-num}
In this section, we study the numerical behaviour of the diffuse-interface LdG model.

To simplify implementation and demonstration, and to facilitate further comparison with the numerical results in literature, we use the 2D reduced LdG model on a re-scaled 2D domain, $\tilde\Omega$. 
The reduced LdG model is the thin film limit of the full LdG model on a 3D domain $\Omega = \tilde\Omega \times [0,h]$, in the $h\to 0$ limit, under certain boundary conditions \cite{canevari_well_2020}. 
Using a special reduced temperature $A=-\frac{B^2}{3C}$, the $\vb Q$-critical point of the LdG free energy \eqref{E-ldg} can be related to a reduced $2\times 2$ symmetric traceless tensor $\vb P$, as follows \cite{han_reduced_2020}
\begin{equation} \label{Q-and-P}
    \vb{Q}\bumpeq\begin{bmatrix}
        \vb{P}+\frac{s_+}{6}\vb{I}_2 & 0 \\ 0 & -\frac{s_+}{3}
    \end{bmatrix}.
\end{equation} This reduced matrix, $\vb P$, is the reduced LdG order parameter.
Using the relation \eqref{Q-and-P}, the corresponding reduced energy is obtained by means of shifting $\bar{E}_{\bar\ve}$ \eqref{ldg-phf-nondim} by an additive constant.
\begin{equation} \label{ldg-phf-nondim-2d}
\begin{aligned}
    \widetilde{E}_{\bar\ve} [\vb{P},\phi]
    &=\int_{\tilde\Omega} \qty[ \frac{1}{2} |\grad\vb{P}|^2 + \bar\lambda^2 \qty(-\frac{B^2}{4C^2} \tr\vb{P}^2 + \frac{1}{4} (\tr\vb{P}^2)^2)] \d x \\
    &\quad+\bar\omega_p\bar\lambda \int_{\tilde\Omega} \qty[ \bar\ve |\nabla\phi|^2+\bar\ve^{-1} \phi^2(1-\phi)^2]\d x 
    +\bar\omega_a \bar\lambda \int_{\tilde\Omega} \bar\ve\left| (\vb{P}+s_+/2)\nabla\phi \right|^2 \d x \\
    &\quad +\bar\omega_v\bar\lambda^2 \int_{\tilde\Omega} \frac12(1-\phi)^2 |\vb P|^2 \d x, \\
    \text{s.t.}\quad & \int_{\tilde\Omega} \phi\d x = \bar V_0.
\end{aligned}
\end{equation}
Recall that $\bar\ve=\ve/\lambda$ is the relative capillary width, and the re-scaled constants are defined by
\[\bar{\lambda}^2 \triangleq \frac{\lambda^2 C}{L},\ 
\bar{\omega}_p = \frac{\omega_p}{\sqrt{CL}},\ 
\bar{\omega}_v = \frac{\omega_v}{C},\ 
\bar{\omega}_a = \frac{\omega_a}{\sqrt{CL}},\]
so the coefficients of the trailing integrals in \eqref{ldg-phf-nondim-2d} are respectively $\frac{\omega_p\lambda}{L},\frac{\omega_a\lambda}{L},\frac{\omega_v\lambda^2}{L}$.
We take $B = \unit[0.64\times 10^4]{N\cdot m^{-2}}$, $C = \unit[0.35 \times 10^4]{N\cdot m^{-2}}$, and $L = \unit[4\times 10^{-11}]{N}$ in the following numerical simulations, which are the physical constants of the commonly used nematic liquid crystal material MBBA \cite{majumdar_equilibrium_2010}. We also fix the following hyperparameters
\[\bar\ve=\frac{\ve}{\lambda}=0.005, \ \bar V_0=0.09, \
\frac{\omega_p}{L}=\unit[3\times 10^7]{m^{-1}},\ \frac{\omega_v}{L}=\unit[6\times 10^{14}]{m^{-2}}.
\]

Using the finite difference method, we discretize the 2D domain $\tilde\Omega=[0,1]\times [0,1]$ with an $N\times N$ grid where $N=128$, and minimize $E_\ve$ numerically following the gradient flow with respect with $\vb Q$ and $\phi$. We study the minimizers of the diffuse-interface LdG energy for various domain sizes, $\lambda$, and anchoring strengths, $\omega_a$, in Figure \ref{t2r}. We plot $\lambda$ on the vertical axis and $\frac{\omega_a}{L}$ on the horizontal axis, and this quantity also has units of length.
For small $\lambda$ and small $\omega_a$, the diffuse-interface energy minimizer is the radial state with circular N-I interface and a $+1$ defect at the centre of the nematic phase. As $\lambda$ increases, the ordering increases inside the nematic droplet and the N-I interface becomes more pronounced. The interface shape deforms from circular to ellipsoidal, as $\lambda$ increases, for fixed $\omega_a$.
As $\omega_a$ increases, for a fixed $\lambda$, the minimizer becomes a polar state with an ellipsoidal interface, and we also see two $+1/2$ defects along the long axis.
As $\omega_a$ increases and for stronger tangential anchoring, the defects are expelled from the N-I interface slightly and move towards the interior of the tactoid.
For large $\lambda$ and large $\omega_a$, the minimizer is a tactoidal state for which the nematic directors are almost parallel to each other, with tangential anchoring on the N-I interface and two $+1/2$-interior defects close to the spherical caps of the tactoid.

We briefly compare our numerical results with existing work. The radial and polar states on a 2D disc with tangential boundary condition have been studied in \cite{hu_disclination_2016}. The radial state is radially symmetric with a central $+1$ defect, and the polar state has two $+1/2$ defects along the diameter.
People have also numerically found the tactoid, which is a spindle-shaped nematic droplet within which directors are almost parallel to each other, and with either defects near the pointy ends of the spindle or none at all \cite{adler_nonlinear_2023,debenedictis_shape_2016}. These numerical observations in published work are consistent with the phase diagram in Figure \ref{t2r}, for small $\lambda$ and $\omega_a$ (radial state), tactoids with tangential anchoring and with no interior defects (large $\lambda$ and small to moderate $\omega_a$), tactoids with tangential anchoring and interior defects (large $\lambda$ and moderate to large $\omega_a$).
\begin{figure}[th]
    \centering
    \begin{tikzpicture}[x=0.1\textwidth,y=0.1\textwidth]
        \node at (0,0)[anchor=south west]{\includegraphics[width=0.6\textwidth]{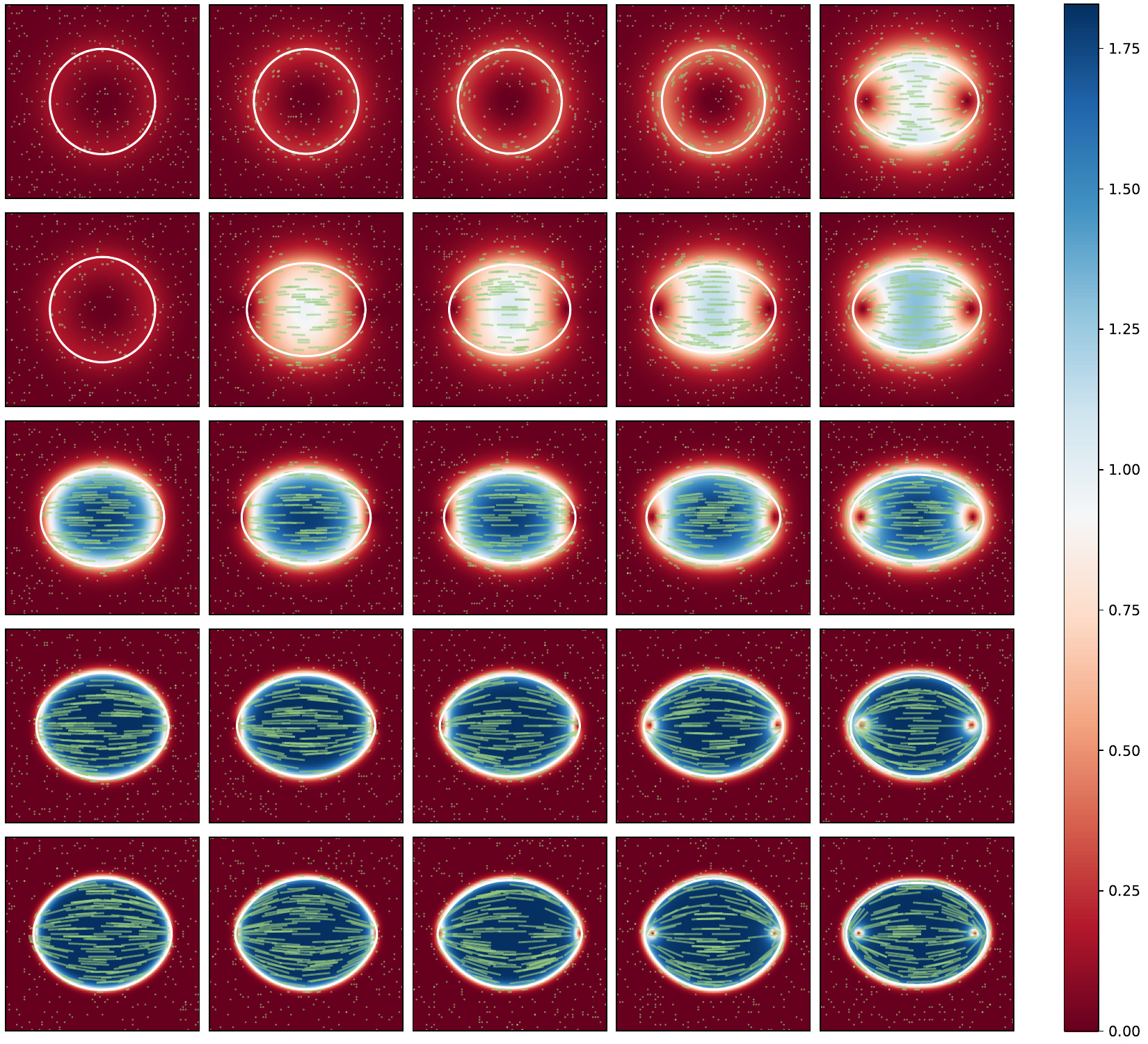}};
        \draw[->,very thick] (1,0)--(5,0);
        \node at (3,0)[anchor=north]{$\omega_a/L$};
        \draw[->,very thick] (0,5)--(0,1);
        \node at (-0.2,3.05)[rotate=-90]{$\lambda$};
    \end{tikzpicture}
    \caption{The plots of the minimizers of \eqref{ldg-phf-nondim-2d}, for various values of $\lambda$ in rows from top to bottom: $0.8,1,2,5,7.5$ (unit: $\unit[10^{-6}]{m}$); values of $\omega_a/L$ in columns from left to right: $1,3,6,9,15,30$ (unit: $10^7 m^{-1}$), with $\bar\ve=0.005$, $\bar V_0=0.09$, $\omega_p/L=\unit[3\times 10^7]{m^{-1}}$, $\omega_v/L=\unit[6\times 10^{14}]{m^{-2}}$.
    The colour bar represents the order parameter $|\vb P|^2$. The white contour along N-I interface denotes the level set $\{\phi=0.5\}$. For the green line, the length represents the positive eigenvalue of $\vb P$ and the direction represents the eigenvector of $\vb P$ corresponding to the positive eigenvalue.}
    \label{t2r}
\end{figure}

\section{Conclusion and discussions} \label{sec:conclusion}
In this paper, we propose a diffuse-interface LdG energy functional, $E_\ve$ \eqref{ldg-phf} for the free boundary between a nematic droplet and an isotropic surrounding. The model incorporates a LdG $\vb Q$-tensor and a phase field $\phi$ as order parameters, with an extra hyperparameter $\ve$ introduced to describe the width of interface.

We establish the existence (Proposition \ref{min-exist}) of global minimizers of $E_\ve$.
For fixed $\phi$, we have also shown that the $\vb Q$-critical points satisfy the maximum principle (Theorem \ref{Q-uni-bound}) and uniqueness (Proposition \ref{min-unique}) under certain conditions.
The existence, uniqueness and regularity properties, as a whole, ensure the solvability of mathematical models centred around the diffuse-interface LdG functional, $E_\ve$.
We also prove the sharp-interface limit of $E_\ve$ when $\ve\to 0$, in the sense of $\Gamma$-convergence (Theorem \ref{sharp-intf-limit} and Proposition \ref{ldgphf-gam-lim}), i.e. convergence of minimizers of $E_\ve$ to minimizers of $E_0$ as $\ve \to 0$.
The sharp-interface limit of the diffuse-interface LdG energy is defined in terms of a nematic region $D$ and the $\vb Q$-tensor order parameter, with contributions from the LdG energy, a void energy in the complement and a boundary energy.
The theoretical work is complemented by numerical tests that yield physically interpretable results, which are also consistent with existing observations in other free-boundary NLC models. In particular, we observe the transition form radial to polar and tactoidal states when adjusting the scale $\lambda$, under different hyperparameters.

There are multiple generalisations of our work.
The diffuse-interface framework is a powerful tool in describing deformable regions, and can be used to model other LC phases such as smectic phases with free boundaries \cite{shi_modified_2024,xia_variational_2023}.
For example, there is room for improvement for our theoretical results such as the maximum principle and uniqueness, which is only proven in the fixed-$\phi$ case under smoothness assumptions. It could be generalized to global minimizers under weaker conditions.
We have only studied energy minimizers in Figure \ref{t2r}. 
There may be multiple energy minimizers and future work includes the study of solution landscapes of \eqref{ldg-phf-nondim} i.e., the study of stable and unstable saddle points of \eqref{ldg-phf-nondim} and the pathways between them, using the method of saddle dynamics \cite{han_multistability_2023,shi_hierarchies_2023,yin_construction_2020}. 
The mesh width of the finite difference method has also limited our investigation into the sharp-interface limit, so more advanced numerical schemes, such as the spectral method, could be implemented to study the rate of convergence of the diffuse-energy LdG energy to its sharp interface limit $\ve \to 0$. 
To summarise, our work is a firm stepping stone into the hugely complex and fascinating field of NLC shape morphologies with free boundaries, and will contribute to mathematical toolboxes for studying optimal shapes for prescribed conditions, and the inverse problem of finding optimal conditions for a prescribed or desired NLC morphology. 

\section*{Acknowledgements}

Dawei Wu and Baoming Shi appreciate the University of Strathclyde for their support and hospitality when work on this paper was undertaken.

\appendix
\numberwithin{theorem}{section}
\section{Calculation of Euler-Lagrange equation} \label{sec-el}

From this section onwards, we return to dropping bars in the nondimensional energy \eqref{ldg-phf-nondim} except the bar over $\bar F_b$.
Consider a perturbation $\delta\vb Q$ to the function $\vb Q\in H_0^1(\Omega;\mathcal{S}_0),$ and the variational derivative $\frac{\delta E_\ve}{\delta \vb Q}$ is an $\mathcal{S}_0$-valued function that satisfies 
\[\delta E_{\ve} = \int_\Omega \frac{\delta E_\ve}{\delta \vb Q} :\delta\vb Q\d x + o(\|\delta\vb Q\|_{L^2}),\]
where $\vb A:\vb B$ is the matrix inner product.
Taking the variational derivative of $E_\ve$ with respect to $\vb Q$ results in the following Euler-Lagrange equation.
\[\frac{\delta E^{\rm LdG}}{\delta\vb Q}
    +\omega_v\lambda^2 \frac{\delta E^{\rm void}}{\delta\vb Q}
    +\omega_a\lambda \frac{\delta E_\ve^{\rm anch}}{\delta\vb Q}=0.\]
The first variation of the mixing energy $E_{\ve}^{mix}$ with respect to $\vb Q$ vanishes.

The variational derivative of the LdG energy $E^{LdG}$ has been computed in \cite[Eq. (3.24)]{majumdar_equilibrium_2010} as follows (notation adapted to our form).
\begin{equation} \label{del-ldg-del-Q}
    \frac{\delta E^{\rm LdG}}{\delta\vb Q}=-\triangle\vb Q + \lambda^2 \nabla_{\mathcal{S}_0} \bar F_b(\vb Q),
\end{equation}
where $\nabla_{\mathcal{S}_0}$ is the projected gradient operator on $\mathcal{S}_0$, which maps the matrix gradient with the projection operator $\mathscr{P}_0$ in \eqref{trace0-proj}.
The projection is necessary because of the traceless constraint on $\vb Q$.
Hence, we have
\[\nabla_{\mathcal{S}_0} \bar F_b(\vb Q)=\mathscr{P}_0 \qty[\qty(\frac{A}{C}+\tr\vb Q^2)\vb Q-\frac{B}{C}\vb Q^2]
=\qty(\frac{A}{C}+\tr\vb Q^2)\vb Q-\frac{B}{C}\qty(\vb Q^2-\frac{\tr \vb Q^2}{3}\vb I ).\]

The variational derivative of the void penalty $E^{\rm void}$ is
\begin{equation}
\frac{\delta E^{\rm void}}{\delta\vb Q}=(1-\phi)^2 \nabla_{\mathcal{S}_0}\qty(\frac12 |\vb Q|^2) = (1-\phi)^2 \vb Q.
\end{equation}

The variational derivative of the anchoring penalty $E^{\rm anch}_{\ve}$ is
\begin{equation}
    \begin{aligned}
        \frac{\delta E_\ve^{\rm anch}}{\delta\vb Q}&=\ve \nabla_{\mathcal{S}_0} \qty[ \left| \qty(\vb Q+\frac{s_+}{3}\vb I)\nabla\phi \right|^2] \\
        &= \ve \mathscr{P}_0\qty[(\nabla\phi\otimes\nabla\phi)\qty(\vb Q+\frac{s_+}{3}\vb I) + \qty(\vb Q+\frac{s_+}{3}\vb I)(\nabla\phi\otimes\nabla\phi)].
    \end{aligned}
\end{equation}
During the computation of the above equation, we have used the relation
\[\left|\qty(\vb Q+\frac{s_+}{3}\vb I)\nabla\phi\right|^2=(\nabla\phi\otimes\nabla\phi): \left( \vb Q+\frac{s_+}{3}\vb I \right)^2,\]
and the fact that if $f(\vb Q)=\vb A:\vb Q^k$ where $k\in\mathbb{N}$, then the matrix gradient of $f$ is
\[\nabla f(\vb Q)=\sum_{j=0}^{k-1}\vb Q^j\vb A \vb Q^{k-1-j}.\]
Adding the expressions above gives \eqref{ldgphf-el-Q}.

Taking the first variation of $E_\ve$ with respect to $\phi$, we get
\[\omega_p\lambda\frac{\delta E_\ve^{\rm mix}}{\delta\phi}+\omega_a\lambda\frac{\delta E_\ve^{\rm anch}}{\delta\phi}+\omega_v\lambda^2\frac{\delta E^{\rm void}}{\delta\phi}+2\omega_p\lambda\ve\xi=0,\]
where $\xi\in\BbbR$ is a Lagrangian multiplier from the volume constraint \eqref{vol-con}.
The first variation of LdG energy $E^{\rm LdG}$ with respect to $\phi$ vanishes.
We can then directly compute that
\begin{align}
    \frac{\delta E_\ve^{\rm mix}}{\delta\phi}
    &=-2\ve \triangle\phi + 2\ve^{-1}\phi(\phi-1)(2\phi-1), \\
    \frac{\delta E_\ve^{\rm anch}}{\delta\phi}
    &=-2\ve \divergence\qty[\qty(\vb Q+\frac{s_+}{3}\vb I)^2 \nabla\phi], \\
    \frac{\delta E^{\rm void}}{\delta\phi}&= (\phi-1) |\vb Q|^2.
\end{align}
Adding the expressions above and dividing the entire equation by $2\omega_p\lambda\ve$ gives us the equation \eqref{ldgphf-el-phi}.

\bibliographystyle{siam}
\bibliography{main.bib}
\end{document}